\documentclass[oneside, reqno]{amsart}

\newcommand{\hide}[1]{}

\usepackage{amssymb}
\usepackage{graphicx}
\usepackage{amsmath}

\usepackage{enumitem}
\usepackage{bm}

\usepackage{mathtools}
\usepackage{graphics}
\usepackage{xcolor}
\usepackage{multicol}

\usepackage{textpos}
\usepackage{subfigure}
\usepackage{framed} 
\usepackage{tikz}
\usepackage{tikz-cd}
\usepackage{commath}
\usepackage{nicematrix}
\usepackage{geometry}
\usepackage{float}

\usepackage[doi=false,isbn=false,url=false,eprint=false, style=alphabetic, backend=bibtex,maxbibnames=99]{biblatex}
\usetikzlibrary{shapes, arrows.meta, positioning}

\usepackage[colorlinks=true,urlcolor=blue]{hyperref}

\DeclareMathAlphabet{\mathbbold}{U}{bbold}{m}{n}

\usepackage[title,titletoc,toc]{appendix}

\theoremstyle{definition}

\newcommand{\ITM}{\mathrm{ITM}}
\newcommand{\IET}{\mathrm{IET}}

\newcommand{\R}{\mathbb{R}}

\renewcommand{\ge}{\geqslant}
\renewcommand{\le}{\leqslant}

\newcommand{\betas}{\beta_*}
\newcommand{\betass}{\beta_{**}}

\numberwithin{figure}{section}
\numberwithin{equation}{section}

\theoremstyle{plain}
\newtheorem{theorem}{Theorem}[section]

\newtheorem*{conjecture*}{Conjecture}
\newtheorem*{mconjecture}{Boshernitzan--Kornfeld Conjecture}
\newtheorem*{mtheoremI}{Main Theorem I}
\newtheorem*{mtheoremII}{Main Theorem II}
\newtheorem*{thm:stable=acc+m}{Theorem \ref{thm:stability=accm}}
\newtheorem*{thm:acc+m-approx-ep}{Theorem \ref{thm:ep->acc+m}}
\newtheorem*{theorem*}{Theorem}
\newtheorem{lemma}[theorem]{Lemma}

\theoremstyle{remark}

\theoremstyle{definition}
\newtheorem{definition}[theorem]{Definition}

\usepackage{xpatch}
\makeatletter
\AtBeginDocument{\xpatchcmd{\@thm}{\thm@headpunct{.}}{\thm@headpunct{}}{}{}}
\makeatother

\newcounter{lstv}

\newcounter{lsta}

\author[Kostiantyn Drach]{Kostiantyn Drach}
\author[Leon Staresinic]{Leon Staresinic}
\author[Sebastian van Strien]{Sebastian van Strien}

\address{Universitat de Barcelona (Gran Via de les Corts Catalanes, 585, 08007 Barcelona, Spain)}
\address{Centre de Recerca Matem\`atica (Edifici C, Carrer de l'Albareda, 08193 Bellaterra, Spain)}
\email{kostiantyn.drach@ub.edu}

\address{University of Z\"urich (190 Winterthurerstrasse, 8057 Z\"urich, Switzerland)}
\email{leon.staresinic@math.uzh.ch}

\address{Imperial College London (180 Queen's Gate, South Kensington, London SW7 2AZ, UK)}
\email{s.van-strien@imperial.ac.uk}

\title[Topological Prevalence of Finite Type Interval Translation Maps]{Topological Prevalence of Finite Type Interval Translation Maps} 
\date{\today} 

\thanks{The first author was partially supported from grants CNS2025-166633 (AEI), PID2023-147252NB-I00 (AEI), CEX2020-001084-M (Maria de Maeztu Excellence program), and the ERC Advanced Grant ``SPERIG'' (\#885707). The second author acknowledges the support from Imperial College London through the Roth PhD scholarship and the Swiss National Science Foundation through grant number TMCG-2\_213663·2023. The third author acknowledges a partial sponsorship via Bj\"orn Winckler's Marie Curie postdoctoral fellowship \#743959.
\newline 
\indent The content of this paper is part of an earlier manuscript \cite{drach2025densitystableintervaltranslation}, which is available on arXiv. That manuscript has been split into a three-part series: \cite{drach2026transversalityintervaltranslationmaps}, \cite{drach2026characterisationstabilityintervaltranslation}, and the present paper.}

\begin{document}

\maketitle

\begin{abstract} An \emph{interval translation map} ($\ITM$) is a map $T \colon I \to I$ defined as a piecewise translation on a finite partition of an interval $I$ into $r \ge 2$ subintervals. Unlike classical interval exchange transformations ($\IET$s), the images of these subintervals are allowed to overlap, making $\ITM$s a natural generalisation of $\IET$s.

An $\ITM$ $T$ is said to be \emph{of finite type} if its attractor $\bigcap_{n\ge 0}  T^n(I)$ is a finite union of intervals; in this case, restricted to this invariant set, $T$ is bijective and hence behaves like an $\IET$. Otherwise, $T$ is \emph{of infinite type}. In this paper, for every $r \ge 2$, we prove that the set of finite type $\ITM$s contains an open and dense subset in the space of all possible $\ITM$s on $r$ subintervals. This confirms a topological version of a long-standing conjecture due to Boshernitzan and Kornfeld. 
\end{abstract}

\section{Introduction}
\label{sec:intro}

\textit{Interval translation maps} ($\ITM$s) are orientation-preserving piecewise isometries of an interval. As such, they are generalizations of the well-known interval exchange transformations ($\IET$s), obtained by dropping the bijectivity assumption on the mapping. They were introduced by Boshernitzan and Kornfeld in \cite{MR1356616}. See Figures~\ref{fig:ITM} and \ref{fig:IET} for comparison examples between an $\ITM$ and an $\IET$.

We fix an interval $I := [0,1)$. More formally, given $r \ge 2$ and two sets of parameters:
\begin{itemize}
    \item[--]
    \emph{discontinuities} $0 = \beta_0 < \beta_1 < \beta_2 < \dots < \beta_{r-1} < \beta_r = 1$ and 
    \item[--] 
    \emph{translation factors} $\gamma_1, \gamma_2, \dots, \gamma_r \in \R$,
\end{itemize}
an interval translation map $T \colon I \to I$ is defined for each $1 \le i \le r$ as
\[
T(x) := x + \gamma_i, \quad x \in [\beta_{i-1}, \beta_i).
\]
The condition $T(I) \subset I$ imposes the constraints $\gamma_i \in [-\beta_{i-1},\, 1 - \beta_i]$ for all $1 \le i \le r$. Therefore, the \emph{parameter space} $\ITM(r)$ of $\ITM$s on $r$ intervals identifies with a convex polytope in $\mathbb{R}^{2r-1}$. We endow $\ITM(r)$ with the induced metric and topology.

\begin{figure}[th]
    \centering
    \begin{minipage}{0.45\textwidth}
        \centering
        \vspace{-1mm}
        \includegraphics[width=1.1\textwidth, trim={18 15 15 10},clip]{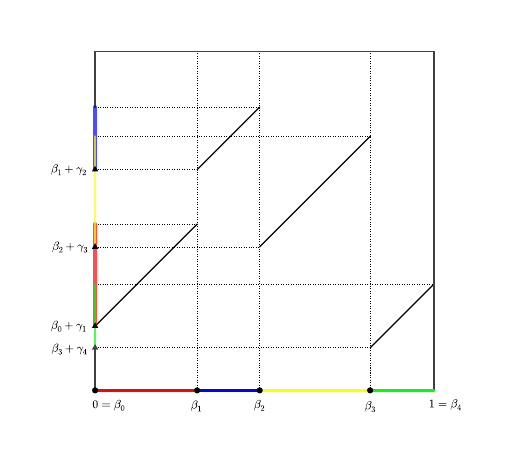}
        \vspace*{-6mm}
        \caption{$\ITM$ on $4$ intervals}\label{fig:ITM}
    \end{minipage}\hfill
    \begin{minipage}{0.45\textwidth}
        \centering
        \vspace{-2mm}
        \includegraphics[width=1.1\textwidth, trim={20 15 20 10},clip]{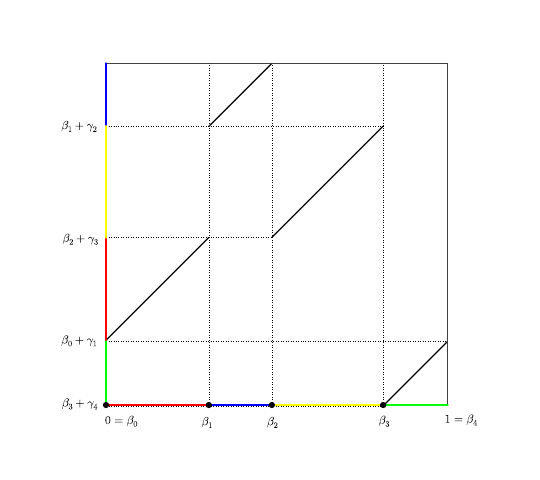}
        \vspace*{-6mm}
        \caption{$\IET$ on $4$ intervals}\label{fig:IET}
    \end{minipage}
\end{figure}

A fundamental distinction between $\ITM$s and $\IET$s is that, for a typical $T \in \ITM(r)$, one has $T(I) \subsetneq I$. Therefore, the sets $X_n := T^n(I)$ form a sequence of sets nested as $I = X_0 \supsetneq X_1 \supseteq X_2 \supseteq X_3 \supseteq \dots$. Each of these sets is a finite union of intervals, and hence their intersection 
\[
X := \bigcap_{n \ge 0} X_n
\]
is a non-empty \emph{attractor} of the system. This further leads to a natural dichotomy: $T$ is said to be of \emph{finite type} if $X_{n+1} = X_n = X$ for some $n$ (and hence for all larger indices), and of \emph{infinite type} if $X_{n+1} \subsetneq X_n$ for all $n$.

In the case $r=2$, all $\ITM$s are of finite type, and the attractor $X$ is a single interval on which the induced map is conjugate to a rotation. In general, it is known that the closure $\overline{X}$ of $X$ in $[0,1]$ decomposes as a disjoint union $A_1 \sqcup A_2$ (one of these sets may be empty), where $A_1$ is a finite union of intervals and $A_2$ is a Cantor set (see Lemma~\ref{lem:x-structure}). Moreover, it was shown in \cite{MR1796167} that $T$ is of finite type if and only if $X$ is a finite union of intervals, and that if $T$ is of infinite type and topologically transitive on $X$, then $\overline{X(T)}$ is a Cantor set. Finally, by \cite{drach2026transversalityintervaltranslationmaps}, the set $\overline{X} \setminus \{1\}$ is the non-wandering set of $T$.

The first explicit example of an infinite type map was constructed in \cite{MR1356616} by Boshernitzan and Kornfeld for $r=3$. They conjectured that such behavior should be rare:

\begin{mconjecture}
\label{conj:inf-type-zero}
For all $r \ge 2$, the set of all infinite type $\ITM$s on $r$ intervals is a measure zero subset of $\ITM(r)$.    
\end{mconjecture}

This conjecture has been the focus of a large part of research in the field of $\ITM$s, but remains widely open, except in a few cases.
Namely, it was first established for $r=3$ and a special $2$-parameter family of $\ITM$s in \cite{MR2013352} by Bruin and Troubetzkoy. This family supports a renormalization scheme analogous to the Rauzy--Veech induction for $\IET$s. The dynamical properties of this renormalization imply that the measure of the set of all infinite type maps is zero, analogous to how the analysis of the Rauzy--Veech induction shows that almost all minimal $\IET$s are uniquely ergodic (\cite{MR644018}, \cite{MR644019}) and weakly mixing (\cite{MR2299743}). Generalising the results for Rauzy--Veech induction used in these proofs, generic unique ergodicity and weakly mixing for infinite type maps in the family from ~\cite{MR2013352} were established in ~\cite{MR4973368} and ~\cite{ artigiani2026typicalweakmixingexceptional}, respectively (see also \cite{bruin2023interval} for related results). Figures \ref{fig:bt-fin} and \ref{fig:bt-inf} show the parameter space of the special family from \cite{MR2013352}: the right figure (taken from \cite{MR2013352}) shows the set of infinite type maps, while the union of couloured regions in the left figure\footnote{The picture is by Bj\"orn Winckler.} is the set of all finite type maps. These triangles are dynamically related through renormalization, see \cite{MR2013352} for details.

\begin{figure}[H]
    \centering
    \begin{minipage}{0.48\textwidth}
        \centering
        \includegraphics[width=0.9\textwidth]{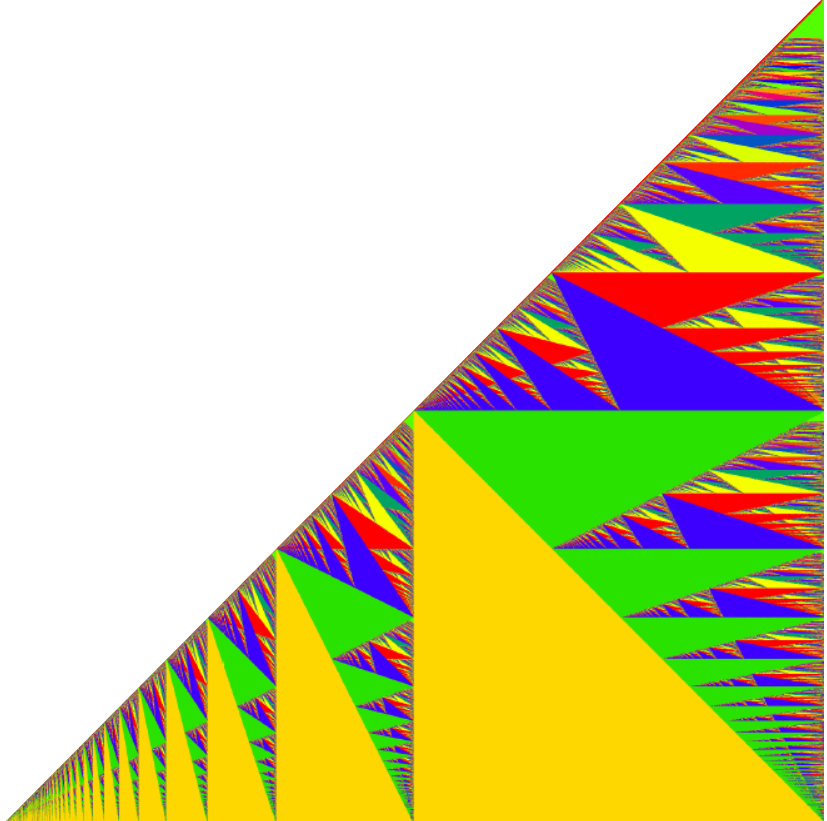}
        \caption{Finite type maps from the family in \cite{MR2013352}}\label{fig:bt-fin}
    \end{minipage}\hfill
    \begin{minipage}{0.48\textwidth}
        \centering
        \includegraphics[width=\textwidth]{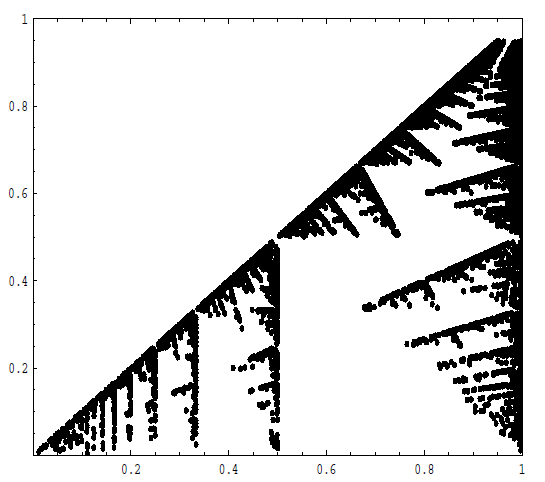}
        \caption{Infinite type maps from the family in \cite{MR2013352}}\label{fig:bt-inf}
    \end{minipage}
\end{figure}

A similar method of renormalization was used for $r \le 4$ and a more complicated family of \textit{double rotations} introduced in \cite{MR2152403}, and studied further in \cite{MR2966738} and \cite{MR4397159}. The conjecture was established in full for $r=3$ in \cite{MR3124735} by showing that almost every $\ITM$ on three intervals can be renormalized to a double rotation. Unfortunately, there has been little progress on developing a renormalization scheme for $\ITM$s on an arbitrary number of intervals (except in some special cases \cite{MR2308208}).

In this paper, instead of using renormalization, we build on the stability theory of $\ITM$s developed in \cite{drach2026characterisationstabilityintervaltranslation} and the transversality results from \cite{drach2026transversalityintervaltranslationmaps} to establish the topological version of the Boshernitzan--Kornfeld Conjecture:

\begin{mtheoremI}[Topological Prevalence of Finite Type Maps]
\label{mtheoremI}
For every $r \ge 2$, the set of all finite type $\ITM$s on $r$ intervals contains an open and dense subset of $\ITM(r)$.    
\end{mtheoremI}

This theorem is a simple consequence of the following more precise result:

\begin{mtheoremII}[Density of Stable Interval Translation Maps]
\label{mtheoremII}
The set $\mathcal{S}(r)$ of all stable interval translations maps on $r$ intervals is a dense subset of $\ITM(r)$.
\end{mtheoremII} 

For the definition and overview of stability, see Subsection~\ref{subsec:stability}. 

\begin{proof}[Main Theorem II $\implies$ Main Theorem I]
By Main Theorem I, the stable maps form a dense subset of $\ITM(r)$. Since the set $\mathcal{S}(r)$ is open by definition (see Definition~\ref{def:stable}) and stable maps are of finite type (see Lemma~\ref{lem:stable-fin-type}), the result follows.
\end{proof}

To prove Main Theorem II, we need to show that the set of stable maps $\mathcal{S}(r)$ forms a dense subset of $\ITM(r)$. This is difficult to show directly, so we adopt the following approach. First, we use the following precise dynamical characterisation of stable maps, which is the main result of \cite{drach2026characterisationstabilityintervaltranslation}:

\begin{thm:stable=acc+m}[Characterisation of Stability]
A finite type map is stable if and only if it satisfies the ACC and Matching conditions.
\end{thm:stable=acc+m}

The definitions of the ACC and Matching conditions are recalled in Subsection~\ref{subsec:stability}. Next, we define a simple type of map, called \textit{eventually periodic maps}, that can be shown to be of finite type (Lemma \ref{lem:ep-fin-type}) and dense in $\ITM(r)$ (Lemma \ref{lem:ep-dense}). Crucially, they can also be approximated by maps satisfying the ACC and Matching conditions:

\begin{thm:acc+m-approx-ep}[ACC + Matching Approximates Eventually Periodic Maps]
For every eventually periodic map $T \in \ITM(r)$, there exists an arbitrarily small perturbation $\Tilde{T} \in \ITM(r)$ of $T$ that is of finite type and satisfies the ACC and Matching conditions.
\end{thm:acc+m-approx-ep}

The proof of this theorem is contained in Section \ref{sec:stab-approx-ep}, and the main ingredient is the Perturbation Lemma from \cite{drach2026transversalityintervaltranslationmaps}, stated here as Lemma~\ref{lem:pert-lem}.

\begin{proof}[Proof of the Main Theorem II]
By the Characterisation of Stability (Theorem \ref{thm:stability=accm}), a map is stable if and only if it satisfies the ACC and Matching conditions. By Theorem \ref{thm:ep->acc+m}, every eventually periodic map is accumulated upon in $\ITM(r)$ by maps satisfying the ACC and Matching conditions. Thus, the closure $\overline{\mathcal{S}(r)}$ of the set of stable maps in $\ITM(r)$ contains the set of eventually periodic maps. Since, by Lemma \ref{lem:ep-dense}, the eventually periodic maps are dense, we conclude $\overline{\mathcal{S}(r)} = \ITM(r)$.
\end{proof}

Main Theorem II also gives the following measure-theoretic result (proved in the Appendix \ref{appendix:ae-fin-stable}): almost every finite type map is stable, Theorem \ref{thm:ae-fin-stab}. Moreover, it also suggests a refinement of the Boshernitzan--Kornfeld Conjecture, in the following sense (see Theorem \ref{thm:bk->irr-rot}): If the Boshernitzan--Kornfeld Conjecture is true, then for almost every $\ITM$, the return map to every connected component of $X$, which must be an interval since $T$ is of finite type, is equivalent to an irrational circle rotation.

The structure of the paper is as follows. In Section \ref{sec:preliminaries} we recall the notation and main definitions for $\ITM$s. We also list the results used in the rest of the paper: basic dynamical facts about $\ITM$s (Subsection~\ref{subsec:basic-res-def}), the theory of stability for $\ITM$s (Subsection~\ref{subsec:stability}, following \cite{drach2026characterisationstabilityintervaltranslation}), and the Perturbation Lemma (Subsection~\ref{subsec:perturbation-lemma}, following \cite{drach2026transversalityintervaltranslationmaps}). In Section \ref{sec:ep-maps}, we introduce eventually periodic maps and establish some of their dynamical properties. In Section \ref{sec:stab-approx-ep} we prove Theorem \ref{thm:ep->acc+m}. In Appendix \ref{appendix:ae-fin-stable}, we prove that almost every finite type map is stable (Theorem \ref{thm:ae-fin-stab}) and establish a refinement of the Boshernitzan--Kornfeld Conjecture (Theorem \ref{thm:bk->irr-rot}).

\subsection{Overview of $\IET$s, $\ITM$s and piecewise isometries}
\label{subsec:history}

The subject of interval exchange transformations ($\IET$s) goes back to 
the 70's. Some the earliest publications were by Katok \cite{MR0331438}, Keane \cite{MR0357739}, \cite{MR0435353} and Veech \cite{MR0516048}. For a broad overview of $\IET$s, we recommend the survey papers \cite{MR2219821}, \cite{MR2648692}.  It is well-known that $\IET$s are connected to billiards, see for example \cite{MR0399423} and \cite{MR0644840}, \cite[Ch. 5]{MR0832433} and \cite[Ch.2, \S 4D]{MR0889254}, and to flows on flat surfaces, see \cite{MR2000471}, \cite{MR644019}, \cite{MR644018}, \cite{MR1393518}, \cite{MR1733872}, \cite{MR2261104}.

In the first part of the introduction, we covered the results about generic behavior and renormalization for $\ITM$s, so we mention the remaining topics here. The topological properties of general $\ITM$s were first studied by Schmeling and Troubetzkoy in \cite{MR1796167}. Buzzi proved in \cite{MR1855837} that all piecewise isometries on polytopes have zero topological entropy (for $\ITM$s this was remarked already in \cite{MR1356616}). Buzzi and Hubert proved in \cite{MR2054049} that all piecewise monotone maps without periodic points (with some necessary assumptions) are semi-conjugate to $\ITM$s (possibly with flips). They also showed that the upper bound for the number of ergodic measures of an $\ITM$ on $r$ intervals is $r-1$, in contrast to $\IET$s for which the bound is $\frac{r}{2}$ (for the precise result see \cite{MR2648692}). This bound was realized by the family considered by Bruin in \cite{MR2308208}. $\ITM$s were connected to billiards with {\lq}one-sided mirrors{\rq}, and analysed in \cite{MR3449199} and \cite{MR3403406}. Connecting $\ITM$s with flows on surfaces is a very interesting open problem.

$\ITM$s also appeared in the work of Levitt in 1993 (see \cite{MR1231840}) in the context of group action on $\mathbb{R}$-trees. There, he considered a two-parameter family of $\ITM$s related to the holonomy map of certain 1-dimensional foliations.

$\ITM$s are a special class of piecewise isometries. There is a vast literature on such systems, and so we mention just a small subset of it: \cite{MR1905204}, \cite{MR1938473}, \cite{MR1772421}, \cite{MR1738947}, \cite{MR2091702}, \cite{MR1992662}, \cite{MR2039048}, \cite{MR2221800}, \cite{MR2486783}, \cite{MR3010377}, \cite{MR4032960}, \cite{MR4075314}, \cite{MR4441154}. In \cite{MR4082258}, the authors connected 1D and 2D piecewise isometries by embedding $\IET$s into planar piecewise isometries.

\section{Preliminaries}
\label{sec:preliminaries}

\subsection{Notation and conventions} 
\label{subsec:not-conv}
In this subsection, we recall some of the main conventions and notation for $\ITM$s. Since the maps we are considering are discontinuous at finitely many points, we will adopt the standard convention of considering every discontinuity $\beta$ of a map $T$ as two points $\beta^- < \beta^+$ and define $T(\beta^-) := \lim_{x \uparrow \beta}$ and $T^{\beta^+} := \lim_{x \downarrow \beta}$. We double the preimages of the point $\beta$ in the same way. The images and preimages of $\beta^-, \beta^+$ will be called \textit{signed points}. In a similar way, we can define $x^- < x^+$, $T(x^-)$, $T(x^+)$, etc.\ for any point $x \in I$.

\begin{definition}
\label{def:touching}
We say that a pair of signed points $(a,b)$ \textit{touches} (or \textit{is touching}) if the set $\{a,b\}$ is equal to $\{x^+, x^-\}$ for some point $x \in I$. In that case, we will write $a \sim b$.
\end{definition}
\noindent The \textit{critical set} $\mathcal{C}$ of $T$ is defined as follows:
\[
\mathcal{C} := \{\beta_1^-, \beta_1^+,\dots,\beta_{r-1}^-, \beta_{r-1}^+\}.
\]
We will refer to the elements of $\mathcal{C}$ as either the critical points or the discontinuities, depending on the context. We will sometimes use the labels $\beta_0^+ := 0^+$ and $\beta_r^- := 1^-$, but we do not consider them as critical points. By $\mathcal{C}^+$, we denote the set of all $+$-type points in $\mathcal{C}$, and by $\mathcal{C}^-$ the set of all $-$-type points in $\mathcal{C}$. 

Most of the time, we do not need to know the index of a discontinuity with respect to the order in $I$ nor whether the discontinuity is of $+$-type or $-$-type. That is why we will often use labels $\beta$, $\betas$ and $\betass$ to denote the discontinuities we are dealing with. If we care about the sign of a discontinuity, we will use the labels $\beta^+, \betas^-$, etc. In that case, we will use the same label without a sign to denote discontinuity of $T$ corresponding to the signed discontinuity, e.g. we denote by $\beta$ the discontinuity such that $\beta^+$ ($\beta^-$) is the $+$-part ($-$-part) of $\beta$.

The perturbation of (or a map sufficiently close to) some starting map $T$ will be denoted by $\Tilde{T}$. Many of the most important objects associated to a map $T$, e.g. critical points, the set $X$ and boundary points of intervals contained in $X$, will have well-defined continuations for sufficiently small perturbations of $T$. For any such object $Z$, we will denote its continuation by $\Tilde{Z}$, e.g. $\Tilde{\beta}^+$, $\Tilde{X}$, $\Tilde{J}$ and similar.

We associate to each point $x \in I$ its \textit{itinerary}. The itinerary of $x$ is an infinite sequence of integers $(i_0(x), i_1(x), \dots, i_n(x), \dots)$, with $1 \le i_n(x) \le r$, where $i_n(x) = s$ means that $T^n(x) \in I_s$. We will often use `itinerary up to time $n$' to refer to the first $n$ elements of this sequence. 

Finally, for a point $x$ we will refer to the set $O(x) := \{ x, T(x), T^2(x), \dots \}$ as the \textit{$T$-orbit} of $x$. Moreover, if $S$ is a subset of $I$, we will refer to the set $O(S) := \bigcup_{x \in S} O(x)$ as the $T$-orbit of $S$. If the map $T$ is clear from the context, we will omit it and simply use the term `orbit'. We will refer to the set $O(x,n) := \{ x, \dots, T^{n-1}{x} \}$, where $n \ge 0$, as the orbit up to time $n$ of $x$, and to the set $O(S,n) := \bigcup_{x \in S} O(x,n)$ as the orbit up to time $n$ of $S$.

\subsection{Basic results and definitions}
\label{subsec:basic-res-def}
In this subsection, we recall some of the main definitions and list the elementary results for $\ITM$s that we use in the paper. For more details and proofs, see \cite{drach2026transversalityintervaltranslationmaps}.

We start with the following useful definitions:

\begin{definition}
\label{def:c1c2}
For $T \in \ITM(r)$, we define the following sets:
\begin{itemize}
    \item $\mathcal{C}_1$ is the set of all $\beta \in \mathcal{C}$ that eventually land on a discontinuity;
    \item $\mathcal{C}_2$ is the set of all $\beta \in \mathcal{C}$ that never land on a discontinuity, but are eventually periodic;
    \item $\mathcal{C}_0 = \mathcal{C}_1 \cup \mathcal{C}_2$;
    \item $\mathcal{C}_i^{\pm} = \mathcal{C}^{\pm} \cap \mathcal{C}_i$, for $i \in \{ 0, 1, 2 \}$.
\end{itemize}
\end{definition}

Recall the definition of a first return map:
\begin{definition}[First return map]
\label{def:rj}
For an interval $J \subset X$, define $R_J$ to be the \emph{first return map} (or simply the \emph{return map}) to $J$ under $T$, i.e.\ for every $x \in J$ that returns to $J$, we define $R_J(x) := T^k(x)$ for the smallest integer $k = k(x) \ge 1$ such that $T^k(x) \in J$.
\end{definition}
The following is a simple lemma about the first return maps for intervals in $X$:
\begin{lemma}
\label{lem:rj-facts}
The domain of $R_J$ is the entire interval $J$. $R_J: J \to J$ is bijective and $J$ is partitioned into finitely many maximal half-open subintervals such that no point in their interiors lands on a discontinuity of $T$. \qed
\end{lemma}

The return map $R_J$ is continuous on these subintervals of $J$, so for simplicity we refer to these intervals as `continuity intervals of $R_J$'. We usually use $x$ and $y$ to denote the boundary intervals of $J$, so that $J = [x,y)$.

We will almost exclusively consider return maps to intervals that are connected components of $X$. It will be convenient to use the following shorthand for such intervals:

\begin{definition}
\label{def:comp-int}
We will say that an interval $J$ is a \textit{interval component} of $X$ if $J$ is equal to a connected component of $X$.
\end{definition}

We use the following notation for such intervals $J$. By Lemma \ref{lem:rj-facts}, there are finitely many points $a_{1}, \dots, a_{N-1}$ in the interior of $J$ that land on discontinuities before returning to $J$. It is also convenient to define $a_{0} := x$ and $a_{N} := y$. The continuity intervals of $R_J$ are denoted by $J_1, \dots, J_N$. The return time of $J_j$ to $J$ is denoted by $r_j$, for $1 \le j \le N$, and the first landing times of $a_j$ to discontinuities of $T$ are denoted by $l_j$, for $1 \le j \le N-1$.

An interval component $J$ for which no point in the interior of $J$ lands on a discontinuity will be called \textit{dynamically trivial}. For such intervals, the return map $R_J$ is the identity. Interval components not of this form will be called \textit{dynamically non-trivial}. Any interval that is referred to as dynamically trivial or a non-trivial is assumed to be an interval component of $X$.

The closure $\overline{X}$ of $X$ in $[0,1]$ has the following structure:

\begin{lemma}
\label{lem:x-structure}
For any map $T$, $\overline{X}$ is equal to $A_1 \cup A_2$, where $A_1$ is a finite union of closed intervals and $A_2$ is a Cantor set. The union is disjoint, except possibly at the right endpoints of the intervals in $A_1$. Moreover, one of the sets $A_1, A_2$ is allowed to be empty. \qed
\end{lemma}

In \cite{drach2026characterisationstabilityintervaltranslation} it was also shown that $\overline{X} \setminus \{1\}$ is equal to the non-wandering set of $T$:

\begin{lemma}
\label{lem:nonwandering} 
$\overline{X} \setminus \{ 1 \}$ is equal to the non-wandering set $\Omega(T)$. \qed
\end{lemma}

Lemma \ref{lem:J-dynamics} gives some elementary dynamical properties for the interval components of $X$, and it will be used implicitly.

\begin{lemma}
\label{lem:J-dynamics}  
Let $T \in \ITM(r)$. Then the following two properties hold: 
\begin{enumerate}[label=(\alph*)]
    \item Every interval component of $X$ is of the form $[T^{k_1}(\beta^+), T^{k_2}(\betas^-))$, respectively, for some $\beta^+, \betas^- \in X \cap \mathcal{C}$ and $k_1,k_2 \ge 0$;
    \item If $T^{l_1}(\beta)$ is an interior point of $X$ for some $\beta \in X \cap \mathcal{C}$ and $l_1 \ge 0$, then there exist $\betas \in X \cap \mathcal{C}$ and $l_2 \ge 0$ such that $T^{l_1}(\beta) \sim T^{l_2}(\betas)$. \qed
\end{enumerate}
\end{lemma}

Finally, Lemma \ref{lem:orb-class} gives a simple classification of the possible $T$-orbit behaviour:  

\begin{lemma}[Orbit Classification Lemma]
\label{lem:orb-class}
For each point $x \in I$, at least one of the following three possibilities holds:
\begin{enumerate}
\item (Precritical) $x$ lands on a discontinuity of $T$;
\item (Preperiodic) $x$ lands on a periodic point of $T$;
\item (Accumulation) $x$ accumulates on a discontinuity of $T$. More precisely, there exist discontinuities $\beta, \betas \in \mathcal{C}$ such that $x$ accumulates on $\beta$ from the left and on $\betas$ from the right. \qed
\end{enumerate}
\end{lemma}

\subsection{Stability, ACC and Matching}
\label{subsec:stability}
In this subsection, we briefly review the stability theory of $\ITM$s. For more details and proofs, see \cite{drach2026characterisationstabilityintervaltranslation}.

\begin{definition}[Stable maps]
\label{def:stable} 
We say that a map $T \in \ITM(r)$ is \textit{stable} if there a neighbourhood $\mathcal{U}$ of $T$ in $\ITM(r)$ such that:

\begin{enumerate}
\item The mapping assigning to each map $\Tilde{T}$ in $\mathcal{U}$ the corresponding set $\overline{X}(\Tilde{T})$ is continuous with respect to the Hausdorff topology on compact subsets of $[0,1]$. Moreover, for each $\tilde{T} \in \mathcal{U}$, $\overline{X}(\Tilde{T})$ is homeomorphic to $\overline{X}(T)$.
\item The number of discontinuities in $I \setminus \overline{X}(\Tilde{T})$ is constant in $\mathcal{U}$.
\end{enumerate}
\end{definition}

The following is a simple lemma:

\begin{lemma}
\label{lem:stable-fin-type}
Stable maps are of finite type. \qed
\end{lemma}

The properties that characterize stability are those that prevent a discontinuous change in the topology of $X$. To describe these properties, we need to recall a few definitions.

\begin{definition}[Ghost preimage]
Let $\betas^+$ a discontinuity of $T$. A discontinuity $\betass^-$ that lands on $\betas^-$ is called a \textit{ghost preimage} of $\betas^+$.
\end{definition}

Thus a ghost preimage of a $+$-type discontinuity $\betas^+$ is a $-$-type discontinuity $\betass^-$ that lands on $\betas^-$. The definition can easily be extended to a $-$-type discontinuity. The motivation for the name is the following: $\betass^-$ is \textit{almost} a preimage of $\betas^+$, i.e.\ by an arbitrarily small perturbation we can make the iterate $T^k(\betass^-)$ land to the right of $\betas^+$, thus creating an actual preimage of $\betas^+$. Ghost preimages of discontinuities form trees:

\begin{definition}[Ghost tree]
\label{def:ghost-tree}
Let $\beta$ be a discontinuity of $T$. The \textit{ghost tree} $\mathcal{GT}(\beta)$ of $\beta$ is defined inductively in the following way. Set $\beta$ to be the root of the tree, i.e.\ the set of level $0$ vertices of the tree. Assume that we have defined all of the vertices of level $\le n$ and all of the edges between them. Then the level $n+1$ vertices of the tree correspond to the set of all ghost preimages (if such exist) of the discontinuities corresponding to level $n$ vertices of the tree. The new edges are those between discontinuities and their ghost preimages, i.e.\ a new direct edge $\betas \leftarrow \betass$ is added if and only if $\betas$ is a level $n$ vertex, $\betass$ is a level $n+1$ vertex and $\betass$ is a ghost preimage of $\betas$. 
\end{definition}

The following is the first of the two properties characterizing stability:

\begin{definition}[Absence of Critical Connections (ACC)]
\label{def:acc}
We say that a finite type $T$ satisfies the \textit{ACC} condition if the following three conditions hold: 
\begin{enumerate}
    \item (A1) For every interval component $J$ of $X$ and each point $a \in J$, the orbit of $a$ up to and not including the return time to $J$ contains at most one critical point of $T$;
    \item (A2) For every dynamically non-trivial interval $J$, none of the boundary points of $J$ land on discontinuities up to and not including the return time to $J$;
    \item (A3) For every discontinuity $\beta \notin X$, its ghost tree $\mathcal{GT}(\beta)$ does not contain $\beta$.
\end{enumerate}
\end{definition}

Roughly speaking, A1 and A2 prevent discontinuous changes that make $\overline{X}$ smaller, while A3 prevents discontinuous changes that make it bigger. For more details and a more formal discussion, see \cite{drach2026characterisationstabilityintervaltranslation}. The other property that characterizes stability is:

\begin{definition}[Matching]
\label{def:matching}
We say that a finite type $T$ satisfies the \textit{Matching} condition if for every dynamically non-trivial interval $J$, exactly one point $a$ in the interior of $J$ lands on a critical point before returning to $J$.
\end{definition}

The reason we call this property Matching is that if $a$ is the single point from above, then we have that:
\begin{enumerate}
    \item $J = [R_J(a^+), R_J(a^-))$;
    \item $R^2_j(a^-) \sim R^2_J(a^+)$,
\end{enumerate}
\noindent and thus the second iterates under $R_J$ of $a^+$ and $a^-$ `match'. With the definitions out of the way, we are ready to state the main theorem of \cite{drach2026characterisationstabilityintervaltranslation}:

\begin{theorem}[Characterisation of Stability]
\label{thm:stability=accm}
A finite type interval translation map $T$ is stable if and only if it is of finite type and satisfies the ACC and Matching properties. \qed
\end{theorem}

\subsection{Perturbation Lemma}
\label{subsec:perturbation-lemma}

In this subsection, we recall an important result from \cite{drach2026transversalityintervaltranslationmaps}, stated below as Perturbation Lemma \ref{lem:pert-lem}. This is the main ingredient in the proof of Theorem \ref{thm:ep->acc+m}. It holds for two different types of intervals: interval components of $X$ (Definition \ref{def:comp-int}) and maximal periodic intervals (Definition \ref{def:max-per-int}).

\begin{definition}
\label{def:max-per-int}
A maximal interval $J \subset I$ consisting of periodic points with the same itinerary is called a \textit{maximal periodic interval}. It is clear that $J$ is a half-open interval and that its boundary points have the property that they land on discontinuities of $T$ or on the boundary points of $I$. Moreover, no point in the interior of $J$ lands on a discontinuity. 
\end{definition}

An interval component of $X$ can be a maximal periodic interval as well, but a maximal periodic interval can be strictly contained in an interval component of $X$. The version of the Perturbation Lemma for a maximal periodic interval is only used in the proof of Lemma \ref{lem:pert-corr}.

We also need the following useful definition:

\begin{definition}
\label{def:trans-fac}
Let $n > 0$ be a natural number. We define the \textit{translation factor} $Tr(x,n)$ at time $n$ of a point $x \in I$ as the translation factor of $T^n$ restricted to $x$, i.e.\ $Tr(x,n) = T^n(x)-x$.
\end{definition}

This definition can easily be extended to intervals $K \subset I$ of points that have the same itinerary up to time $n > 0$, by calling the translation factor $Tr(K,n)$ at time $n$ of $K$ the translation factor of any point $x$ in $K$. Note that $Tr(K,n) = 0$ if and only if $K$ is periodic with period $n$. We denote the translation factors defined for a perturbation $\Tilde{T}$ of $T$ by $\Tilde{T}r$.

The Perturbation Lemma \ref{lem:pert-lem} has two parts, which can be roughly described as follows. The first states that all sufficiently small dynamical changes (of a certain type) of a return map $R_J$ to an interval component $J$ of $X$ are achievable by perturbations of the underlying map $T$ that depend continuously on the size of the dynamical changes. The second states that the perturbation can also be chosen so that we control two types of critical connections: those in the orbit of $J$ up to return time and those outside the orbit of $J$.

\begin{lemma}[Perturbation Lemma]
\label{lem:pert-lem}
Let $T$ be an interval translation map such that $T(I)$ is compactly contained in the interior of $I$. Let $J$ be one of the following: an interval component of $X$ or a maximal periodic interval. Then there exists an $\epsilon_0 > 0$ depending on $J$ and $T$, with the following properties. For every $\epsilon < \epsilon_0$ and every choice of $\epsilon^{\gamma}_1, \dots \epsilon^{\gamma}_{N}, \epsilon^{\beta}_0, \dots \epsilon^{\beta}_{N} \in (-\epsilon, \epsilon)$ there exists a perturbation $\Tilde{T}$ such that $|\Tilde{T} - T| \to 0$ as $\epsilon \to 0$, and the following holds:

\begin{enumerate}
    \item There exists an interval $\Tilde{J} \subset I$ that is $\epsilon$-close to $J$ and partitioned into intervals $\Tilde{J}_j = [\Tilde{a}_{j-1},\Tilde{a}_j)$, with $1 \le j \le N$, such that $\Tilde{J}_j$ maps forward continuously up to time $r_j$ under the iterates of $\Tilde{T}$ and has the same itinerary up to time $r_j$ as $J_j$ for all $1 \le j \le N$. In the case when $J$ is an interval component of $X$, we may set $\Tilde{a_j} - a_j = \epsilon^{\beta}_j$ for all $0 \le j \le N$;
    \item The difference between the translation factors of $\Tilde{J}_j$ and $J_j$ is $\epsilon^{\gamma}_j$ for all $1 \le j \le N$, i.e.\ $\Tilde{T}r(\Tilde{J}_j,r_j) - Tr(J_j,r_j) = \epsilon^{\gamma}_j$.
\end{enumerate}
Moreover, we have that:

\begin{enumerate}[label=(\alph*)]
    \item Let $T^n(\beta^+) = \betas^+$ be a critical connection such that $\beta^+$ and $T^n(\beta^+)$ are both contained in the $T$-orbit of $J_j$ up to time $r_j$ for some $1 \le i \le N$. We may assume that $\Tilde{\beta}^+$ is still contained in the $\Tilde{T}$-orbit of $\Tilde{J}_j$ up to time $r_j$ and the difference $\Tilde{T}^n(\Tilde{\beta}^+) - \Tilde{\beta}_*^+$ can be chosen arbitrary in $[0,\epsilon)$. Analogously for critical connections $T^n(\beta^-) = \betas^-$ and the difference $ \Tilde{\beta}_*^- - \Tilde{T}^n(\Tilde{\beta}^-)$;
    \item For every critical connection $T^n(\beta) = \betas$, with $\beta,\betas \notin O(J)$, such that either $\beta \notin X$ or $\beta, \betas$ are part of a single periodic orbit, the difference $\tilde T^n(\tilde \beta) - \tilde \betas$ can be chosen arbitrary in $(-\epsilon,\epsilon)$. \qed
\end{enumerate}
\end{lemma}

In (a), the changes are such that the itineraries of the critical points are preserved, while this is not necessarily the case in part (b). The assumption that $T(I)$ is compactly contained in the interior of $I$ holds for every $\ITM$ in a complement of finitely many hyperplanes contained in the boundary of $\ITM(r)$, so this may be assumed without loss of generality, and it allows us not to include several necessary assumptions on the $\epsilon, \epsilon^{\gamma}_1, \dots \epsilon^{\gamma}_{N}, \epsilon^{\beta}_0, \dots \epsilon^{\beta}_{N}$ that guarantee that the perturbed map still has image contained in $[0,1)$. 

\section{Eventually periodic maps}
\label{sec:ep-maps} 
In this section, we define eventually periodic maps and analyse their dynamical properties.

\begin{definition}[Eventually periodic maps]
\label{def:ep-maps}
We say that $T$ is \textit{eventually periodic} if every point $x \in I$ eventually lands on a $T$-periodic point.
\end{definition}

Eventually periodic maps are the simplest non-trivial maps, an analogue of Morse--Smale systems from smooth interval dynamics (see \cite{MR1239171}). They are an important example of maps that are dense in the parameter space:

\begin{lemma}[Eventually periodic maps are dense]
\label{lem:ep-dense}
The set of eventually periodic maps forms a dense subset of $\ITM(r)$ for all $r \ge 1$.
\end{lemma}

This lemma is a consequence of the simple observation that the maps $T$ with rational $\gamma$ parameters are eventually periodic. 

\begin{proof}
Since $\ITM(r)$ is a polytope in $\mathbb{R}^{2r-1}$, the points with rational coordinates are clearly dense in it. We now show that the maps corresponding to such points are eventually periodic. In fact, it is enough that only the coefficients $\gamma_1, \dots, \gamma_r$ are rational. Let $\mathbb{N}(\gamma)$ be the set of all linear combinations of $\gamma_1, \dots, \gamma_r$ with integer coefficients. This set is discrete, because the minimal distance between any of its elements, if non-zero, is at least $\frac{1}{q^r}$, where $q$ is the largest denominator of the numbers $\gamma_1, \dots, \gamma_r$. Thus for any $x \in I$, the set $x + \mathbb{N}(\gamma) \cap [0,1)$ is finite, which means that $x$ must eventually land on a periodic point.
\end{proof}

An important characterisation of eventually periodic maps is the following:

\begin{lemma}
\label{lem:disc-ep->map-ep}
Let $T$ be a map for which every point in $\mathcal{C}$ is eventually periodic. Then the $T$ is eventually periodic as well.
\end{lemma}

\begin{proof}
We first prove that there exists an $\epsilon > 0$ such that for each $\beta^+ \in \mathcal{C}^+$, resp. $\beta^- \in \mathcal{C}^-$, the interval $[\beta, \beta+\epsilon)$, resp. $[\beta-\epsilon,\beta)$, is eventually periodic and maps forward continuously for all times. We prove this for $+$-type discontinuities, with the other case being analogous.

Let $N$ be the smallest time such that $T^N(\beta^+)$ is a periodic point for all $\beta^+$, and let $p$ be the maximal period of these points. Let $\epsilon > 0$ be the closest positive distance the orbit of any $\beta^+$ up to time $N+p$ gets to the left of a point in $\mathcal{C}^+$. Note that $\beta^+$ can land on another critical point, but in our definition of $\epsilon$ we ignore these times and look only at the positive distances. We claim that this is our required $\epsilon$. Indeed, if the iterates of any interval $[\beta, \beta + \epsilon)$ contain a discontinuity before time $N+p$, it would mean that the distance between a point in the orbit of $\beta^+$ and another critical point would be less than $\epsilon$, contradicting the definition of $\epsilon$. Thus, every such interval maps forward continuously up to time $N+p$. Since $T^N(\beta^+)$ is periodic and has period $\le p$, we see that the entire interval $[T^N(\beta), T^N(\beta) + \epsilon)$ consists of points of the same period, which shows that $\epsilon$ has the required property.

By the Orbit Classification Lemma, each point $x \in I$ is either periodic, lands on a discontinuity, or accumulates on a discontinuity. By assumption, every point in $\mathcal{C}$ is eventually periodic, so every $x$ that lands on a discontinuity is also eventually periodic. If $x$ accumulates on some discontinuity $\beta$ then it must eventually enter an $\epsilon$-neighbourhood of $\beta$. By the previous paragraph, this means that it is eventually periodic, and thus every point $x \in I$ is eventually periodic. 
\end{proof}

Because of Lemma \ref{lem:disc-ep->map-ep}, we have that $\mathcal{C} = \mathcal{C}_0$ if and only if $T$ is eventually periodic. An important property of eventually periodic maps is that they are of finite type:

\begin{lemma}
\label{lem:ep-fin-type}
Eventually periodic maps are of finite type.
\end{lemma}  

\begin{proof}
Let $\epsilon$ and $N$ be as in the proof of Lemma \ref{lem:disc-ep->map-ep}. We will prove that there exists an $M \ge 1$ such that the $T$-orbit of every $x \in I$ enters an $\epsilon$-neighbourhood of some discontinuity $\beta$ in less than $M$ iterates. Then every point enters a periodic interval of $T$ after less than $M+N$ iterates. This implies that $T^{M+N}(I) = T^{M+N+1}$, as they must both be equal to the union of all periodic intervals of $T$, and $T$ is therefore of finite type.

Let $x$ be a point in one of the intervals $I_s$. By assumption, $\beta_{s-1}^+$ is eventually periodic. We define a finite sequence of eventually points $x_0 = \beta_{s-1}^+ < x_1 < \dots < x_t < x$ and times $n_0 = 0 < n_1 < \dots < n_t$ such that $T^{n_i}(x_i) \in \mathcal{C}$ and $d(T^{n_i}(x_i), T^{n_{i-1}}(x_{i-1})) \ge \epsilon$ for all $1 \le i \le t$ by the following inductive procedure. 

If $x_i$ is already $\epsilon$-away from $x$, we stop the procedure. Otherwise, define $n_{i+1}$ to be the minimal time the orbit of the interval $[x_i, x)$ contains a critical point. Such a time exists because $x$ and $x_i$ have different itineraries, as $x_i$ is eventually periodic, while $x$ is not. Let $x_{i+1}$ be the rightmost point in $[x_i, x)$ which lands on a critical point at time $n_{i+1}$. We have that $n_{i+1} > n_i$, because $x_i$ is the rightmost point in $[x_{i-1}, x)$ which lands on a critical point by time $n_i$, which means that the interval $[x_i,x)$ maps forward continuously up to time $n_i$. Moreover, $n_{i+1} - n_i \le N + p$, as $x_i$ becomes periodic with period at most $p$, after at most $N$ iterations. We claim that $x_{i+1} - x_i \ge \epsilon$. Indeed, by construction, $T^{n_i}([x_i,x)) = [\beta, y)$, for some $\beta^+ \in \mathcal{C}^+$. The interval $[\beta, \beta + \epsilon)$ maps forward continuously for all times, which means that only the points in $[\beta + \epsilon, y)$ can land on critical points. Thus $x_{i+1} - x_i \ge \epsilon$, which means that the inductive procedure terminates at some finite time.

By construction, the interval $[x_t,x)$ maps forward continuously up to time $n_t$ when $x_t$ lands on some critical point $\beta^+$. At this time, $x$ enters the $\epsilon$ neighbourhood of $\beta$ and therefore gets mapped into a periodic interval after at most $N$ iterates. Since $n_{i+1} - n_i \le N + p$, we have that $n_t \le (N+p)\lfloor \frac{\beta_{s}-\beta_{s-1}}{\epsilon} \rfloor < M$, for a sufficiently large $M$ that is independent of $x$, which is what we wanted to prove.
\end{proof}

In the proof of Theorem \ref{thm:ep->acc+m}, it will be important to know that a map $T$ remains of finite type and eventually periodic under perturbation. Lemma \ref{lem:ep-fin-type} and Lemma \ref{lem:disc-ep->map-ep} give us a simple criterion for both of these properties: if all of the discontinuities of $T$ remain eventually periodic.

\section{Eventually periodic maps can be approximated by stable maps}
\label{sec:stab-approx-ep} 

\subsection{Overview of the proof of Theorem \ref{thm:ep->acc+m}}
\label{subsec:ep->acc+m-proof-overview}

In this subsection, we review the proof of the following theorem:

\begin{theorem}
\label{thm:ep->acc+m}
Let $T$ be an eventually periodic map. Then arbitrarily close to $T$ in the parameter space, there is a stable map $\Tilde{T}$.  
\end{theorem}

The strategy of the proof is to produce a finite sequence $T_1, \dots, T_k$, for some $k > 0$, of arbitrarily small finite type perturbations of a given eventually periodic map $T$, such that the last map $T_k$ in this sequence is stable. The sequence will be produced in such a way that for each $0 \le i < k$, the so-called \textit{unstable number} (Definition \ref{def:unst-num}) of $T_{i+1}$ is strictly smaller than the unstable number of $T_{i}$. When the unstable number is $0$ and the map satisfies the \textit{correspondence property} (Definition \ref{def:corr}), then the map satisfies A1, A2, and Matching (Lemma \ref{lemma:corr+u=a1a2m}). An additional perturbation then gives A3 as well, which makes the map stable. 

The unstable number is lowered by removing critical connections and boundary discontinuities from $X$. This can seemingly be achieved by simple perturbations of the parameters $(\gamma \, \beta)$, but as always, the main difficulty is in controlling the global dynamics under this perturbation. One of the main difficulties is the fact that there is no good general criterion for when a map $T$ is of finite type, except if it is eventually periodic (Lemma \ref{lem:ep-fin-type}).

To demonstrate this, consider Figure \ref{fig:pert to create a hole}, in which a perturbation has been chosen so that an interval $J_{\epsilon}$ (marked in \textcolor{black}{red}) whose forward iterates contain a discontinuity $\beta$ of $T$ is no longer in the image of the return map to an interval component $J$ of $X$. This implies that $J_{\epsilon}$, and therefore $\beta$, is no longer in $X$, as shown in Figure \ref{fig:propagation of a hole}.

\begin{figure}[h]
    \centering
    \includegraphics[width=0.9\linewidth]{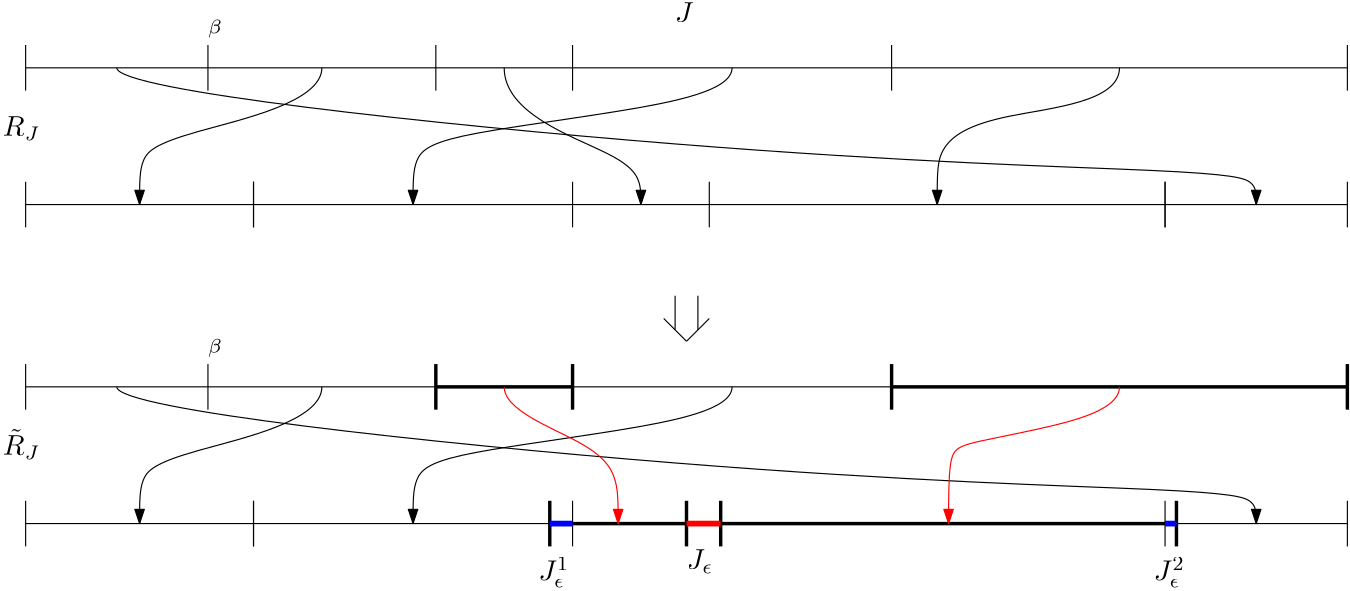}
    \caption{Perturbing a map to create a hole interval $J_{\epsilon}$ marked in red. Forward $\tilde{R}_J$-iterates of $J_{\epsilon}$ contain $\beta$, which is therefore removed from $X$ by this perturbation, resulting in a smaller unstable number. This is depicted in Figure \ref{fig:propagation of a hole}}
    \label{fig:pert to create a hole}.
\end{figure}

The same should then hold for the iterated images $\Tilde{R}_J(J_{\epsilon}), \Tilde{R}^2_J(J_{\epsilon}), \dots$ of $J_{\epsilon}$, except possibly at the intersection of these images with the subintervals $J^1_{\epsilon}$ and $J^2_{\epsilon}$ (marked in \textcolor{black}{blue}). This is because they have multiple $\Tilde{R}_J$-preimages (two in the case of Figure \ref{fig:pert to create a hole}), and some of them might still be contained in $X$ after this perturbation. At the moment, there is no good criterion available that prevents these iterated images from propagating through the entire interval $J$, as is suggested in Figure \ref{fig:propagation of a hole}, which would result in an infinite type map.

\begin{figure}[h]
    \centering
    \includegraphics[width=0.9\linewidth]{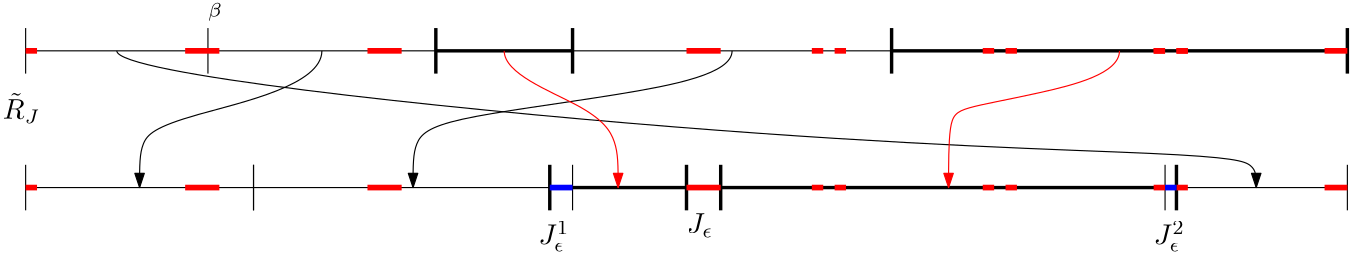}
    \caption{The first few $\tilde{R}_J$ iterates of $J_{\epsilon}$ are drawn in red, indicating that they are no longer contained in $\tilde X$. At some point, these iterates intersect $J^2_{\epsilon}$, which is drawn in blue, and are not necessarily removed from $ \tilde X$, so we do not mark them and their forward iterates in red.}
    \label{fig:propagation of a hole}
\end{figure}

One case in which a perturbation can be chosen so that this problem does not arise is when iterates of a single periodic interval cover the entire interval $J$. In that case, every parameter of the return map $R_J$ (viewed as an $\IET$) is a rational multiple of the length $|P|$ of $P$. Then it can be shown that if the perturbed return map $\Tilde{R}_J$ still satisfies that all of its parameters are rational multiples of $|P|$, then the resulting restriction $\Tilde{R}_{\vert J}$ is of finite type (see the proof Theorem \ref{thm:ep->acc+m} for details). That this is the case is guaranteed by the correspondence property (Definition \ref{def:corr}), which can be achieved after an arbitrarily small perturbation that keeps the map eventually periodic (Lemma \ref{lem:pert-corr}).

To produce these perturbations, we repeatedly use the Perturbation Lemma \ref{lem:pert-lem}. We use 1.,2. and (a) to change the return map around the preimages of an appropriately chosen discontinuity $\beta$ that is removed from $X$ with this perturbation, which results in a strict lowering of the unstable number of $T$. We use (b) to simply keep all of the critical connections outside of the orbit of $J$, which preserves the orbits of all periodic discontinuities outside of the orbit of $J$.

\subsection{Unstable number and correspondence}

In this subsection, we prove a few preliminary lemmas needed in the proof of Theorem \ref{thm:ep->acc+m} in the next subsection. Our main tool will be the Perturbation Lemma \ref{lem:pert-lem}. We first introduce a few concepts related to the dynamics of the discontinuities of an eventually periodic map $T$.

\begin{definition}[Cycle of discontinuities]
Let $T$ be an eventually periodic map. Let $\beta \in X$ be a periodic discontinuity of $T$. The cycle of $\beta$ is defined as:
\[
C(\beta) = O(\beta) \cap \mathcal{C},
\]
i.e.\ the set of all discontinuities in the orbit of $\beta$.
\end{definition}

We say that a cycle is non-trivial if $C(\beta) \neq \{\beta\}$. In this case, $\beta$ lands on at least one other discontinuity of $T$.

Assuming $T$ is eventually periodic, for all $\betas, \betass \in X$ we have that $\betas \in C(\betass)$ if and only if $\betass \in C(\betas)$. Thus the set of all discontinuities contained in $X$ is partitioned into cycles, which we consider as equivalence classes of the relation $\betas \simeq \betass$ if and only if $\betas \in O(\betass)$. Let $\mathcal{Z}$ be the set of all equivalence classes, i.e.\ the set of all cycles.

Because of Theorem \ref{thm:stability=accm}, stable eventually periodic maps have the property that the size of each cycle is exactly one. Such a map also has no discontinuities in the boundary of $X$. With this in mind, we define a number which tells us how far an eventually periodic map is from having this property:

\begin{definition}[Unstable number of $T$]
\label{def:unst-num}
Let $T$ be an eventually periodic map. Then the \textit{unstable number} of $T$, denote by $U(T)$, is defined as:
\[
U(T) \coloneqq \left( \sum_{C \in \mathcal{Z}} (|C|-1) \right) + |\partial X \cap \mathcal{C}|
\]
\end{definition}

For a periodic discontinuity $\beta$, we will denote by $P(\beta)$ the maximal periodic interval containing $\beta$, in the sense of Definition \ref{def:max-per-int}. For brevity, we usually omit stating maximality when discussing periodic intervals, but it is always assumed.

\begin{definition}[Correspondence]
\label{def:corr}
Let $T$ be an eventually periodic map. We say that $T$ has the \textit{correspondence property} if for each discontinuity $\beta^+ \in X$, we have that $\beta^-$ eventually lands into the maximal periodic interval $P(\beta^+)$ and if analogously for every discontinuity $\betas^- \in X$, $\betas^+$ eventually lands into $P(\betas^-)$.
\end{definition}
In other words, if one part (either $+$ or $-$) of a discontinuity is periodic, and therefore is in a boundary of a periodic interval, then the other part must land into this periodic interval. Note that we do not assume whether the other part of the periodic discontinuity is in $X$ or not - both are allowed. Since the orbits of $\beta^-$ and $\beta^+$ both end up in the same periodic cycle of intervals, these points have `corresponding dynamics'. This is a fairly strong property. One of the things it implies is the following restriction on the structure of every interval $J$ of $X$:

\begin{lemma}
\label{lemma:corr-rj-structure}    
Let $T$ be an eventually periodic map that has the correspondence property. Then for every interval $J$ of $X$, there exists a discontinuity $\beta^+ \in X$ such that $J$ is contained in the union of iterates of $P(\beta^+)$.
\end{lemma}

Since $T$ is eventually periodic, each interval of $X$ already has the property of being `tiled' by periodic intervals. If we additionally assume correspondence, the lemma above says that there exists a $\beta^+ \in X$ such that all of these periodic intervals are iterates of $P(\beta^+)$.

\begin{proof}
Assume the contrary, that there is an interval $J$ of $X$ containing iterates of two periodic intervals with disjoint orbits. Let $P_1$ and $P_2$ be two touching periodic intervals in $J$ with this property. By maximality, the point at which they touch eventually lands on a discontinuity $\beta$. Thus $\beta^-$ and $\beta^+$ belong to periodic intervals with disjoint orbits, which contradicts the correspondence property.
\end{proof}

A map that has unstable number zero and satisfies the correspondence property is almost stable:

\begin{lemma}
\label{lemma:corr+u=a1a2m}
Let $T$ be an eventually periodic map. Assume that $U(T) = 0$ and that $T$ satisfies the correspondence property. Then $T$ also satisfies the A1, A2, and Matching properties.
\end{lemma}

\begin{proof}
Since $U(T) = 0$, no point in $X$ can land on two discontinuities, so A1 follows.

By the Lemma \ref{lemma:corr-rj-structure}, any interval $J$ of $X$ is tiled by intervals that are all iterates of a single maximal periodic interval $P(\beta^+)$. Since $U(T) = 0$, $\beta^+$ does not land on any other discontinuity. Since $P(\beta^+)$ is a maximal periodic interval, there exists a discontinuity $\betas^- \in X$ such that the right boundary point of $P(\beta^+)$ lands on $\betas^-$ at some time $k$. Additionally, $P(\betas^-) = T^k(P(\beta^+))$.

Assume that the left boundary point of a dynamically non-trivial interval $J$ lands on a discontinuity before returning to $J$. The case for the right boundary point is analogous. This discontinuity must be $\beta^+$. Since the interval is dynamically non-trivial, a point in the interior of $J$ also lands on a discontinuity of $T$ before returning to $J$. This discontinuity has to be either $\beta$ or $\betas$. In the first case, we get a contradiction with the fact that iterates of the components of a return map are disjoint up to the return time. In the second case, we get that $\beta^+$ lands on $\betas^+$, which is a contradiction with $U(T) = 0$. Thus both cases are impossible, so A2 follows.

If a dynamically non-trivial interval $J$ does not satisfy the Matching property, then there are two points in the interior of $J$ that land on a discontinuity before returning to $J$. These discontinuities can again be only $\beta$ or $\betas$, and we get a contradiction in the same way as for A2. Thus the Matching follows as well.
\end{proof}

Our goal now is to produce a finite sequence of sufficiently small perturbations, each of which decreases the unstable number of $T$, and also makes $T$ satisfy the correspondence property. By induction, this will result in a map that has unstable number zero and satisfies the correspondence property. We then do a final perturbation to get a map that also satisfies A3, which by Lemma \ref{lemma:corr+u=a1a2m} means it satisfies ACC and Matching. 

The main idea of the proof is to remove the discontinuities that contribute to the unstable number of $T$, i.e.\ are in a cycle inside of $X$ or in the boundary of $X$, or violate the correspondence property, while preserving the dynamical properties of points not in the orbit of these discontinuities. 

This result allows us to show that an arbitrarily small perturbation of $T$ satisfies the correspondence property and does not increase the unstable number:

\begin{lemma}[Perturbation to correspondence]
\label{lem:pert-corr}
Let $T$ be an eventually periodic map that does not satisfy the correspondence property. Then there exists an arbitrarily small perturbation $\Tilde{T}$ of $T$ that is eventually periodic, satisfies $U(\Tilde{T}) \le U(T)$ and the correspondence property.
\end{lemma}

\begin{proof}
Assume that $T$ does not satisfy the correspondence property at some discontinuity $\beta$. Without loss of generality, assume that $\beta^+ \in X$, with the other case being analogous. Let $J$ be the periodic interval $P(\beta^+)$, and let $p$ be the minimal period of $\beta^+$. For a sufficiently small $\epsilon$, we apply the Perturbation Lemma \ref{lem:pert-lem}
to produce $\Tilde{T}$ with the following properties. Using 1., we set $\Tilde{T}r(\Tilde{\beta}^+, p) = -\epsilon$, which is possible since $Tr(\beta,p) = 0$, as $\beta$ is periodic with period $p$. This means that $\beta$ now lands $\epsilon$ to the left of itself at time $p$. We preserve all of the connections in (a) and (b),
which means that all of the periodic orbits of $T$ disjoint from $J$ persist for $\Tilde{T}$.

We may also choose $\epsilon$ sufficiently small so that the size of the perturbation is small enough that the itinerary of every discontinuity $\betas \in \mathcal{C}_2$ does not change until its landing time $l(\betas)$ into $X$. Thus for sufficiently small $\epsilon$, each $\betas \in \mathcal{C}_2$ that does not land into $O(J)$ still lands into the same periodic interval after perturbation, at the same time $l(\betas)$. This in particular holds for $\beta^-$, since $T$ by assumption does not satisfy the correspondence property at $\beta$. Then for a sufficiently small $\epsilon$, the entire interval $[\beta-\epsilon,\beta)$ still maps forward continuously up to time $l(\betas)$, and at this time it lands into a periodic interval whose orbit is disjoint from $\Tilde{J}$. Thus $\Tilde{\beta}^+$ is not contained in $\Tilde{X}$. Moreover, the same holds for every discontinuity that lands onto $\beta$ or into $\Tilde{J}$, as every point in $\Tilde{J}$ eventually gets mapped into $[\tilde{\beta}-\epsilon,\Tilde{\beta})$.

By construction, $\Tilde{T}$ is eventually periodic. This is because every discontinuity that landed into $O(J)$ now lands into the same period interval as $\Tilde{\beta}^-$, while every discontinuity with orbit disjoint from $O(J)$ still lands into the same periodic interval. This also implies $U(\Tilde{T}) \le U(T)$. 

Thus by induction, after a finite number of arbitrarily small perturbations, we have that $\Tilde{T}$ is eventually periodic, satisfies the correspondence property at every discontinuity and $U(\Tilde{T}) \le U(T)$.
\end{proof}

\subsection{Proof of Theorem \ref{thm:ep->acc+m}}

As mentioned before, our strategy is to remove from $X$ the discontinuities that are contained in non-trivial critical cycles. The perturbation we make depends on the location of the critical value corresponding to the critical point we want to remove. More precisely, we want the critical value to not be the first or the last one in an interval $J$ of $X$. Such a critical value can be found if the number of continuity intervals of $J$ is at least $4$. Intervals $J$ with fewer than $4$ continuity intervals all require different perturbations, leading to $4$ cases.

\begin{proof}[Proof of Theorem \ref{thm:ep->acc+m}]
By Lemma \ref{lem:pert-corr}, we may assume that $T$ satisfies the correspondence property. Assume first that there is a $\beta^+ \in X$ that lands on at least one other discontinuity, with the proof for $\beta^-$ being analogous. Let $J$ be the interval component of $X$ whose return map has the maximal number of continuity intervals among the interval components containing such a $\beta^+$. Let $N$ be the number of continuity intervals of the return map to $J$. If $N \ge 2$, we may assume that $\beta^+$ is in the interior of $J$. Indeed, if it is in the boundary, we can replace $\beta^+$ by the first critical point in the orbit of $J$ that is contained in the interior of an iterate of $J$. Recall that $\sigma$ denotes the permutation corresponding to the order in which the intervals of $J$ return to $J$, and that $\tau := \sigma^{-1}$.

\underline{Case 1 : $N > 3$}

Let $v_2$ be the second critical value of the return map $R_{J}$ to $J$, with respect to the order inside of $J$ and let $J_{\tau(2)}$ and $J_{\tau(3)}$ be the two continuity intervals of the return map such that the images of their boundary points touch at $v_2$. Let $a_p$, with $1 \le p \le N-1$, be the first critical point of $R_{J}$ in the orbit of $v_2$, which exists because $J$ is, by Lemma \ref{lemma:corr-rj-structure}, equal to the union of iterates of a single periodic interval. Let $P \ge 0$ be minimal time such that $R_{J}^P(v_2) = a_p$. Let $a$ and $b$ be the integers:
\begin{itemize}
    \item $a:= \# \{ R_{J}^t(v_2) \in J_{\tau(2)}; 0\le t < P \}$;
    \item $b := \# \{ R_{J}^t(v_2) \in J_{\tau(3)}; 0\le t < P \}$.
\end{itemize}
Let $\epsilon_1, \epsilon_2 > 0$ be sufficiently small such that for $b > 0$ we have that:
\[
\frac{a}{b+1} < \frac{\epsilon_2}{\epsilon_1} < \frac{a+1}{b},
\]
or if $b = 0$, such that:
\[
a \epsilon_1 < \epsilon_2.
\]
In both cases, $\epsilon_1$ and $\epsilon_2$ satisfy:
\[
-\epsilon_2 < -a \epsilon_1 + b \epsilon_2 < \epsilon_1.
\]
We may choose $\epsilon_1$ and $\epsilon_2$ arbitrarily small. Moreover, we may choose $\epsilon_1$ and $\epsilon_2$ to be of the form $r_1 |P(\beta^+)|$, $r_2 |P(\beta^+)|$, where $r_1$ and $r_2$ are sufficiently small rational numbers.

For $\epsilon$ sufficiently small, we apply the Perturbation Lemma \ref{lem:pert-lem}
to produce $\Tilde{T}$ with the following properties. By 1., we may set $\Tilde{a}_j = a_j$ for all $0 \le j \le N_0$. Using 2., we set $\Tilde{T}r(\Tilde{J}_{\tau(2)}, r_{\tau(2)}) - Tr(J_{\tau(2),r_{\tau(2)}}) = -\epsilon_1$ and $\Tilde{T}r(\Tilde{J}_{\tau(3)}, r_{\tau(3)}) - Tr(J_{\tau(3),r_{\tau(3)}}) = \epsilon_2$. We do not change the translation factors for any of the other $\Tilde{J}_j$ with $j \neq \tau(2), \tau(3)$. This means that $\Tilde{a}^-_{\tau(2)}$ lands $\epsilon_1$ to the left of the point $v_2$, while $\Tilde{a}^+_{\tau(3)-1}$ lands $\epsilon_2$ to the right. We preserve all of the connections in (a) and (b), which means that all of the periodic orbits of $T$ disjoint from $J$ persist for $\Tilde{T}$.

As we do not change any of the $a_j$'s, we have that $\Tilde{J}_j = J_j$ for all $1 \le j \le N$, and in particular $\Tilde{J} := J$. As we do not change the translation factors for any of the other $\Tilde{J}_j$ with $j \neq \tau(2), \tau(3)$, we still have that $\Tilde{R}_J \subset \Tilde{J}$.
The resulting perturbation for the return map with $5$ continuity intervals and the associated permutation $\sigma = (5 3 2 1 4)$ is shown in Figure \ref{fig:perturbation of a return map}.

Similarly to the proof of Lemma \ref{lem:pert-corr}, we may choose $\epsilon$ sufficiently small so that the size of the perturbation is small enough that the itinerary of every discontinuity $\betas \in \mathcal{C}_2$ does not change until its landing time into $X$. Again, this implies that each $\betas \in \mathcal{C}_2$ that does not land into $O(J)$ still lands into the same periodic interval after perturbation. Thus we only need to analyse $\Tilde{O}(\Tilde{J})$.

\begin{figure}
    \centering
    \includegraphics[width=\linewidth]{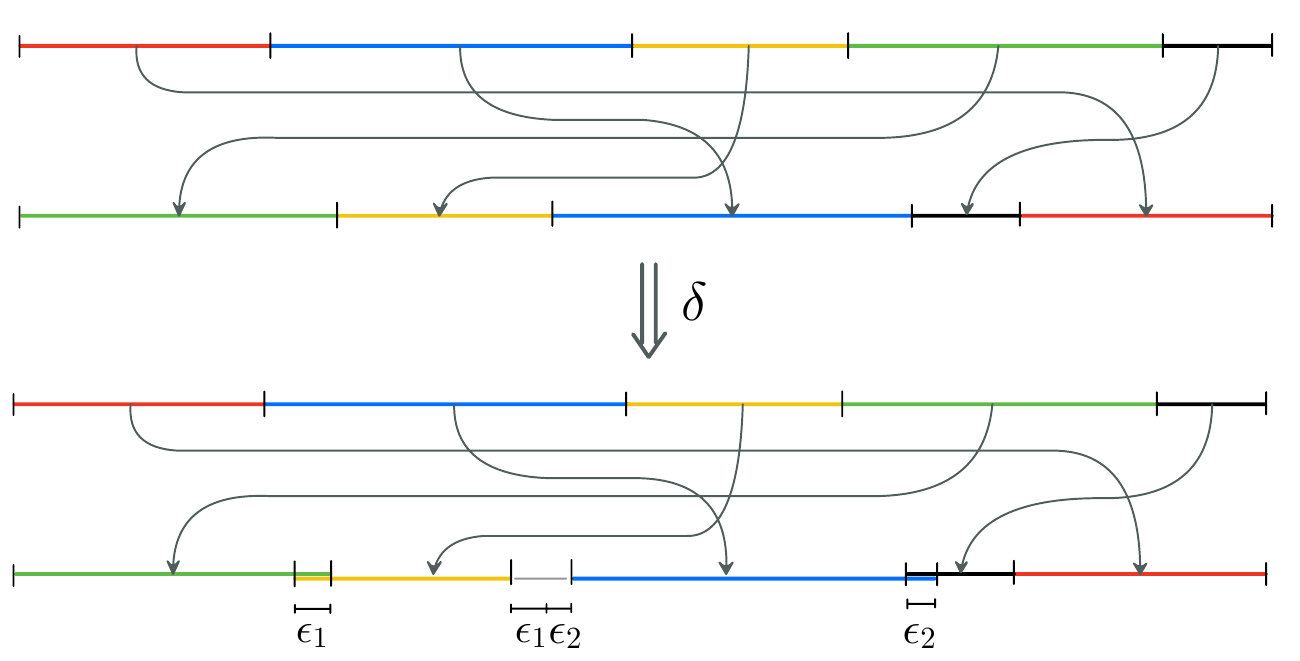}
    \caption{The change in $R_J$ after applying the perturbation in the Case $N > 3$ of the proof of Theorem \ref{thm:ep->acc+m} to a particular interval component $J$ of $X$ for which the return map has $5$ continuity intervals.}
    \label{fig:perturbation of a return map}
\end{figure}

We now show that the first discontinuity $\beta(p)$ that $a_p$ lands on before returning to $J$ is no longer in $X$ after the perturbation. Let us denote by $\Tilde{R}_{J}$ the $\Tilde{T}$ return map to $\Tilde{J}$, which is well-defined by construction. Since $\Tilde{R}_{J}(\Tilde{J}) \subset \Tilde{J}$, we have that $\Tilde{X} \cap \Tilde{O}(\Tilde{J})$ is equal to the $\Tilde{T}$-orbit of $\Tilde{X} \cap \Tilde{J})$ and $\Tilde{X} \cap \Tilde{J}$ = $\bigcap_{n=1}^{\infty} \Tilde{R}^n_{J}(\Tilde{J})$. The interval $[v_2-\epsilon_1,v_2+\epsilon_2)$ is not in $\Tilde{X}$, since it is not in the image of $\Tilde{R}_{J}$ and $\Tilde{X} \cap \Tilde{J}$ = $\bigcap_{n=1}^{\infty} \Tilde{R}^n_{J}(\Tilde{J})$. On the other hand, the intervals $[v_1-\epsilon_1, v_1)$ and $[v_3,v_3+\epsilon_2)$ now both have two $\Tilde{R}_{J}$-preimages, and every other point in $J$ only has a single $\Tilde{R}_{J}$-preimage. As $a_p$ is, by assumption, the first critical point in the orbit of $v_2$, we know that $v_2$ lands on it before it lands on $v_1$ or $v_3$, as they are critical values. Thus for sufficiently small $\epsilon_1$ and $\epsilon_2$, $[v_1-\epsilon_1, v_1)$ and $[v_3,v_3+\epsilon_2)$ are not contained in the $\Tilde{R}_J$-iterates up to and including time $P$ of $[v_2-\epsilon_1,v_2+\epsilon_2)$. Thus the $\Tilde{T}$-iterates of $[v_2-\epsilon_1,v_2+\epsilon_2)$ up to the $P$'th entry time of $v_2$ into $\Tilde{J}$ are not contained in $\Tilde{X}$. The image $\Tilde{R}^P_{J}([v_2-\epsilon_1,v_2+\epsilon_2))$ is by construction equal to the image $R_{J}([v_2-\epsilon_1,v_2+\epsilon_2))$ shifted by $-a \epsilon_1 + b \epsilon_2$. Thus by our choice of $a$, $b$, $\epsilon_1$ and $\epsilon_2$, $\Tilde{a}_p = a_p$ is still contained in this image. Thus $\Tilde{a}_p$ is not contained in $\Tilde{X}$.

The interval $[\Tilde{a}_{p-1},\Tilde{a}_{p+1})$ maps forward continuously up to the time $l$ that $\Tilde{a}_p$ lands on $\Tilde{\beta}(p)$. This means that the $\Tilde{T}^l([v_2-\epsilon_1,v_2+\epsilon_2))$ is disjoint from every $\Tilde{T}^k(\Tilde{J}_j)$, for every $j \neq p,p+1$ and $0 \le k < r_j$. Since $\Tilde{\beta} \subset \Tilde{T}^l([v_2-\epsilon_1,v_2+\epsilon_2))$, this implies $\Tilde{\beta} \notin \Tilde{X}$.

Assuming that $\Tilde{T}$ is eventually periodic, this shows that $U(\Tilde{T}) < U(T)$. Indeed, by construction, all of the critical cycles disjoint from $O(J)$ are preserved for $\Tilde{T}$. Similarly to the proof of Lemma \ref{lem:pert-corr}, the critical points that landed into $O(J)$, but were not contained in $O(J)$ (and therefore not in $X$) are still not contained in $\Tilde{X}$. Since $\Tilde{\beta} \notin \Tilde{X}$, the number of critical points in $\Tilde{O}(J_0) \cap \Tilde{X}$ is smaller by at least one compared to $O(J) \cap X$. Since we assumed $\Tilde{T}$ is eventually periodic, all of the remaining critical points in $\Tilde{O}(J_0)$ are eventually periodic. Some of them are possibly not periodic anymore, and there could be several cycles instead of a single one, but by the definition of the unstable number \ref{def:unst-num}, this would only further decrease $U(\Tilde{T})$. Thus $U(\Tilde{T}) < U(T)$, as claimed.

Finally, we prove that $\Tilde{T}$ is eventually periodic. By construction, $\Tilde{R}_{J_0}$ is conjugate, via the affine map that sends $J$ to $[0,1)$, to an $\ITM$ on $N$ intervals for which all parameters are rational. Indeed, by Lemma \ref{lemma:corr-rj-structure}, $\Tilde{J} = J$ is equal to the union of iterates of $P(\beta^+)$, so the length of every interval $\Tilde{J}_j = J_j$ is an integer multiple of $|P(\beta^+)|$, and by our choice of $\epsilon_1, \epsilon_2$, every translation factor associated to every continuity interval of $\Tilde{R}_{J}$ is also a rational multiple of $|P(\beta^+)|$. This means that every point in $\Tilde{J}$ is eventually periodic. Moreover, every critical point that lands into $\Tilde{O}(\Tilde{J})$ is therefore also eventually periodic. Thus every critical point of $\Tilde{T}$ is eventually periodic, which by Lemma \ref{lem:disc-ep->map-ep} implies that $\Tilde{T}$ is eventually periodic.

\underline{Case 2 : $N = 3$}

This case is different because we only have two critical values of $R_{J}$ and thus we are unable to make the same perturbation as in Case 1 and still have that $\Tilde{R}_{J} \subset \Tilde{J}$ (with $\Tilde{R}_{J}$ defined as in Case 1). The perturbation in this case depends on the permutation $\sigma$ associated to $R_{J}$. We give the proof for $\sigma = (3 2 1)$, with the proof in every other case being analogous. In this case $v_1^- = R_{J}(y^{-})$ and $v_1^+ = R_{J}(a_1^{+})$.

Without loss of generality, we may assume that the first critical value of $R_{J}$ in the backward $R_{J}$-orbit of $\beta$ is $v_1$, with the other case being analogous. Let $P$ be such that $R_{J}^P(v_1) = \beta$, and $l$ such that $T^l(v_1) = \beta$. Define:

\begin{itemize}
    \item $a:= \# \{ R_J^t(v_1) \in J_3; 0\le t < P \}$;
    \item $b := \# \{ R_J^t(v_1) \in J_2; 0\le t < P \}$.
\end{itemize}
Once again, let $\epsilon_1, \epsilon_2 > 0$ be sufficiently, so that for $b > 0$:
\[
\frac{a}{b+1} < \frac{\epsilon_2}{\epsilon_1} < \frac{a+1}{b},
\]
or if $b = 0$:
\[
a \epsilon_1 < \epsilon_2.
\]
In both cases, we again have that:
\[
-\epsilon_2 < -a \epsilon_1 + b \epsilon_2 < \epsilon_1.
\]
Again, we may choose $\epsilon_1$ and $\epsilon_2$ to be of the form $r_1 |P(\beta^+)|$, $r_2 |P(\beta^+)|$, where $r_1$ and $r_2$ are sufficiently small rational numbers.

For $\epsilon$ sufficiently small, we apply the Perturbation Lemma \ref{lem:pert-lem}
to produce $\Tilde{T}$ with the following properties. By 1., we may set $\Tilde{a}_j = a_j$ for all $0 \le j \le 3$. Thus $\Tilde{J}_j = J_j$ for $1 \le j \le 3$, and $\Tilde{J} = J$. Using 2., we set $\Tilde{T}r(\Tilde{J}_{1}, r_{1}) - Tr(J_{1}, r_{1}) = -\epsilon_1$ and $\Tilde{T}r(\Tilde{J}_{2}, r_{2}) - Tr(J_2, r_2) = \epsilon_2$, and do not change $Tr({J}_3, r_3)$. This means that the intervals $[\Tilde{y}-\epsilon_1, \Tilde{y})$ and $[v_1, v_1+\epsilon_2)$ are not contained in the image $\Tilde{R}_J(\Tilde{J})$, while the points in the intervals $[v_2-\epsilon_1,v_2+\epsilon_2)$ have two $\Tilde{R}_J$-preimages. We preserve all of the connections in (a) and (b), which means that all of the periodic orbits of $T$ disjoint from $J$ persist for $\Tilde{T}$. Similar to Case 1., this implies that the dynamics of critical points that do not land into $O(J)$ remain the same, so we only need to analyse $\Tilde{O}(\Tilde{J})$.

By construction, $\Tilde{R}_{J}(\Tilde{J}) \subset \Tilde{J}$ and $\Tilde{R}_{J}: \Tilde{J} \to \Tilde{J}$ is conjugate to an $\ITM$ with rational parameters, so $\Tilde{T}$ is eventually periodic by Lemma \ref{lem:disc-ep->map-ep}. Similarly to Case 1, to prove that $U(\Tilde{T}) < U(T)$, it is enough to show that $\Tilde{\beta} \notin \Tilde{X}$.

Since $[\Tilde{y}-\epsilon_1, \Tilde{y})$ does not have a $\Tilde{R}_J$-preimage, it is not in $\Tilde{X}$. Its image $\Tilde{R}_J([\Tilde{y}-\epsilon_1, \Tilde{y})) = [v_1-\epsilon_1,v_1)$ is therefore also not contained in $\Tilde{X}$, so
$[v_1-\epsilon_1, v_1+\epsilon_2) \notin \Tilde{X}$. By our choice of $v_1$, we know that the $T$-iterates of $v_1$ land onto $\beta$ before landing onto $v_2$. Thus for sufficiently small $\epsilon, \epsilon_1, \epsilon_2$, the $\Tilde{T}$-iterates of $[v_1-\epsilon_1, v_1+\epsilon_2]$ up to and including time $l$ are not contained in $\Tilde{X}$. Recall that $\Tilde{T}^l([v_1-\epsilon_1, v_1+\epsilon_2]) = \Tilde{R}_{J}^P([v_1-\epsilon_1, v_1+\epsilon_2])$, and that the interval $\Tilde{R}_{J}^P([v_1-\epsilon_1, v_1+\epsilon_2])$ is equal to the interval $R_{J}^P([v_1-\epsilon_1, v_1+\epsilon_2])$ shifted by exactly $-a \epsilon_1 + b \epsilon_2$. By our choice of $\epsilon_1$ and $\epsilon_2$, $\Tilde{T}^l([v_1-\epsilon_1, v_1+\epsilon_2])$ therefore contains $\Tilde{\beta}$, so $\Tilde{\beta} \notin \Tilde{X}$.

\underline{Case 3 : $N = 2$}

By our choice of $J$, we know that $\beta^+$ is contained in the interior of $J$ and that it lands on at least one other critical point before returning to $J$. Let $\betas^+$ be the first critical point in the $T$-orbit of $\beta$, and let $l$ be the landing time. Since $N=2$, $\beta$ is the only point in the interior of $J$ that lands on a critical point before returning to $J$.

For $\epsilon$ sufficiently small, we apply the Perturbation Lemma \ref{lem:pert-lem}
to produce $\Tilde{T}$ with the following properties. By 1., we set $\Tilde{\beta} = \Tilde{a}_1 = a_1 = \beta$, $\Tilde{a}_0 = a_0$ and $\tilde{a}_2 = a_2$. Thus $\Tilde{J}_j = J_j$ for $1 \le j \le 2$, and $\Tilde{J} = J$. Using 2., we set $\Tilde{T}r(\Tilde{J}_{1}, r_{1}) - Tr(J_{1}, r_{1}) = -\epsilon$ and $\Tilde{T}r(\Tilde{J}_{2}, r_{2}) - Tr(J_2, r_2) = \epsilon$. This means that the intervals $[\Tilde{x}, \Tilde{x}+\epsilon)$ and $[\Tilde{y}-\epsilon, \Tilde{y})$ are not in the image $\Tilde{R}_J(\Tilde{J})$, and are therefore not in $\Tilde{X}$. 

We preserve all of the critical connections in (a), except the landing of $\beta^+$ onto $\betas^+$, where we set $\Tilde{T}^l(\Tilde{\beta}) - \Tilde{\betas} = \epsilon$. This implies that all critical points $\beta$ landed on before returning to $J$ now land onto $\Tilde{x}^+$, and are therefore not in $\Tilde{X}$. In particular $\Tilde{\beta}_*^+ \notin \Tilde{X}$ Using (b), we preserve all of the critical connections outside of $O(J)$, which similarly as in Cases 1 and 2 implies that the dynamics outside of $O(J)$ do not change, so we only need to analyse $\Tilde{O}(\Tilde{J})$.

We may choose $\epsilon$ to be equal to a sufficiently small rational multiple of $|P(\beta^+)|$. This implies that the return map $\Tilde{R}_J$ to $\Tilde{J}$ is conjugate to an $\ITM$ on two intervals with rational parameters, so $\Tilde{T}$ is eventually periodic by Lemma \ref{lem:disc-ep->map-ep}. As $\Tilde{\beta}_*^+ \notin \Tilde{X}$, we have that $U(\Tilde{T}) < U(T)$.

\underline{Case 4 : $N = 1$}

In this case, $\beta^+$ is in the boundary of $J$, and $J = P(\beta^+)$, since $N=1$ is the maximal number of continuity intervals of the return map for an interval component of $X$ that contains a critical point that lands on at least one other critical point. Let $\betas^+$ be the first discontinuity $\beta^+$ lands on before returning to $J$ and let $l$ be the landing time.

Since $T$ satisfies the correspondence property, $\beta^-$ eventually lands into $J$. Let $r'_1$ be the landing time of $\beta^-$ into $\Tilde{J}$ and let $r'_2 := r_1$ be equal to the return time of $\beta$ to $J$. 

For $\epsilon$ sufficiently small, we apply the Perturbation Lemma \ref{lem:pert-lem}
to produce $\Tilde{T}$ with the following properties. Using 1., we set $\Tilde{\beta} = \Tilde{a}_0 = a_0 = \beta$ and $\Tilde{a}_1 = a_1$, which again implies $\Tilde{J} = J$. Using 2., we set $\Tilde{T}r(\Tilde{J}, r_1) - Tr(J, r_1) = -\epsilon$. This means that $\Tilde{T}^{r_1}(\Tilde{\beta}^+)$ is $\epsilon$ to the left of $\Tilde{\beta}$. We preserve all of the critical connections in (a), except the landing of $\beta^+$ onto $\betas^+$, where we set $\Tilde{T}^l(\Tilde{\beta}) - \Tilde{\betas} = \epsilon$. This implies that all critical points $\beta$ landed on before returning to $J$ now land $2\epsilon$ to the left of $\Tilde{\beta}$. Using (b), we preserve all of the critical connections outside of $O(J)$, which, similar to other cases, implies that the dynamics outside of $O(J)$ do not change, so we only need to analyse $\Tilde{O}(\Tilde{J})$.

For sufficiently small $\epsilon$, the interval $[\Tilde{\beta}-2\epsilon, \Tilde{\beta})$ still lands into $\Tilde{J}$ at time $r'_1$. Since $\Tilde{T}^{r_1}(\Tilde{\beta}^+) = \Tilde{\beta}-\epsilon$, this means that the return map to the interval $[\Tilde{\beta}-\epsilon, \Tilde{T}(\Tilde{\beta}^-))$ is well defined and isomorphic to a rotation. We may choose $\epsilon$ equal to a rational multiple of $|P(\beta^+)|$, which implies that $\Tilde{T}$ is eventually periodic, by Lemma \ref{lem:disc-ep->map-ep}. Moreover, all critical points $\beta$ landed on before returning to $J$ land into $[\Tilde{\beta}-\epsilon, \Tilde{T}(\Tilde{\beta}^-))$ and are therefore not contained in $\Tilde{X}$. In particular, $\Tilde{\beta}_*^+ \notin X$, so $U(\Tilde{T}) < U(T)$.

Thus by induction, we may assume that $T$ satisfies the correspondence property and that there are no critical points in $X$ that land on different critical points. Assume now that there exists a discontinuity $\beta^+$ in the boundary of some interval component $J$ of $X$, with the case for $\beta^-$ being analogous. Since $T$ satisfies the correspondence property, $J = P(\beta^+)$, because otherwise $\beta^+$ would have to land on some other discontinuity. Then we may make exactly the same perturbation as in Case $N=1$, except that now there are no connections to change or keep in part (a) of the Perturbation Lemma \ref{lem:pert-lem}. Thus if $r_1$ be the time at which $\beta^-$ lands into $J_0$, and $r_2$ the minimal period of $\beta^+$, then their return map to the interval $J' = [\Tilde{T}^{r_2}(\Tilde{\beta}^+), \Tilde{T}^{r_1}(\Tilde{\beta}^-))$ is now well defined and isomorphic to a rational rotation. Thus, $\Tilde{T}$ is still eventually periodic and there is one less discontinuity in the boundary of $\Tilde{X}$. We have also added $\beta^-$ to $\Tilde{X}$, but by construction, it is the only discontinuity in its cycle, and therefore $U(\Tilde{T}) < U(T)$.

Thus by induction, we may assume that $T$ is eventually periodic, satisfies the correspondence property and has unstable number zero. By Lemma \ref{lemma:corr+u=a1a2m}, it also satisfies properties A1, A2 and Matching. Assume now that it does not satisfy A3. This in particular means that $C_{\neg X} \neq \emptyset$. For an $\epsilon$ sufficiently small, we apply the Perturbation Lemma \ref{lem:pert-lem} to produce a perturbation $\Tilde{T}$ as follows. We preserve all the critical connections inside of $X$, which are by assumption all of the form  $T^p(\beta) = \beta$, where $p$ is the minimal period of $\beta$, and all critical connections of the form $T^n(\beta) = \betas$, where $\beta \notin X$ and $\betas \in X$. For every critical connection $T^n(\beta^+) = \betas^+$, resp. $T^n(\beta^-) = \betas^-$, such that $\beta, \betas \notin X$, we set $\Tilde{T}^n(\Tilde{\beta}^+) - \Tilde{\beta}_*^+ = \epsilon$, resp. $\Tilde{T}^n(\Tilde{\beta}^+) - \Tilde{\beta}_*^+ = \epsilon$. As before, we may choose $\epsilon$ sufficiently small so that there are no other dynamical changes.

As $\Tilde{T}$ by construction has no critical connections outside of $\Tilde{X}$, A3 must hold (Definition \ref{def:acc}). Since the dynamics of all other critical points do not change, A1, A2 and Matching still hold, so $\Tilde{T}$ is stable.
\end{proof}

\appendix
\section{Almost every finite type $\ITM$ is stable}
\label{appendix:ae-fin-stable}

In this appendix, we prove that the stable maps form a full measure subset of the set of all finite type maps, Theorem \ref{thm:ae-fin-stab} and derive an important consequence: a refinement of the Boshernitzan--Kornfeld Conjecture, Theorem \ref{thm:bk->irr-rot}.

Before getting into the proofs, we need to recall the product notation and dynamically defined vectors, introduced in \cite{drach2026transversalityintervaltranslationmaps}.

For each  $T$,  $x \in I$, $s=1,\dots,r$ and $n\ge 1$, let:
\[
k_s(x,n,T) := \# \{ T^j(x) \in I_s; 0\le j < n \}.
\]
Thus $k_s(x,n,T)$ represents the number of entries of $x$ into $I_s$ up to time $n$. When the map $T$ is clear from the context, we will usually simply write $k_s(x,n)$.

Let $W(r) := \R{}^{r} \oplus \R{}^{r-1}$ be the $(2r-1)$-dimensional real vector space that contains the \textit{coefficient vectors} (introduced below), and let $(\bm{e}_1, \dots, \bm{e}_r)$ and $(\bm{f}_1, \dots, \bm{f}_{r-1})$ be the canonical bases for $\R{}^{r}$ and $\R{}^{r-1}$, respectively. Recall that $\ITM(r)$ is the parameter space of $\ITM$s on $r$ intervals, and that it is a convex polytope contained in $\mathbb{R}^{2r-1}$. We call the elements of this space \textit{parameter vectors}. These vectors also have canonical coordinates coming from the ambient space $\mathbb{R}^{2r-1}$. We will use the shorthand $(\gamma \, \beta)$ for a parameter vector $(\gamma_1 \dots \gamma_r \, \beta_1 \dots \beta_{r-1})$.

Let $\langle \cdot, \cdot \rangle$ be the standard scalar product on $\mathbb{R}^{2r-1}$. Since $W(r) = \R{}^{r} \oplus \R{}^{r-1}$ and $\ITM(r)$ is a subset of $\mathbb{R}^{2r-1}$, it makes sense to write $\langle v, (\gamma \, \beta) \rangle =  \sum_{s=1}^r v_s \gamma_s + \sum_{s=1}^{r-1} v_{s+r} \beta_s$ for a coefficient vector $v = \sum_{s=1}^r v_s \bm{e}_s + \sum_{s=1}^{r-1} v_{s+r} \bm{f}_s \in W(r)$ and a parameter vector $(\gamma \, \beta) \in \ITM(r)$. We call $\langle v, (\gamma \, \beta) \rangle$ the \textit{product} of $v$ and $(\gamma \, \beta)$.

Let $J = [x^{+},y^{-}]$ be an interval component (Definition \ref{def:comp-int}) of $X$. Recall that the return map $R_J$ to $J$ is well-defined (Lemma \ref{lem:lin-dep-rj}) and that there are finitely many points $a_{1}, \dots, a_{N-1}$ in the interior of $J$ that land on discontinuities before returning to $J$. Let $J_1, \dots, J_{N}$ be the continuity interval of $R_J$, and let $r_j$ be the return time of $J_j$ to $J$. Let $a_{0}^{+} := x^+$ and $a_{k}^{-} := y^-$ be the boundary points of $J$. For each $1 \le j < N$, let $m_{j}^{+}$ be the number of discontinuities that $a_{j}^{+}$ lands on before returning to $J$ and, for $1 \le k \le m_{j}^{+}$, let $\beta^{+}(j,k)$ be the $k$-th discontinuity along the orbit up to return time to $J$ of $a_j^{+}$. Define $m_{j}^{-}$ and $\beta^{-}(j,k)$ analogously, for $1 \le j < N$ and $1 \le k \le m_{j}^{-}$. For each $1 \le j < N$, the points $\beta^{+}(j,1)$ and $\beta^{-}(j,1)$ are the $+$ and $-$ part of a single discontinuity, which we denote by $\beta(j)$.

In what follows, it will be useful to denote by $\text{ind}(\beta) \in \{1, \dots, r-1\}$ the index of a discontinuity $\beta$ with respect to this order inside $I$.

We now define three types of dynamically defined vectors associated to the orbit of $J$. The first ones are the \textit{first landing vectors}, that correspond to the first time the points $a_1^J, a_2^J, \dots, a_{N_J-1}^J$ land on discontinuities of $T$:

\begin{definition}[First landing vectors]
\label{def:lan-vec}
For $1 \le j < N$, let $l_{j}$ be the landing time of $a_{j}$ to $\beta(j)$ and let $L_{j} \in W(r)$ be the associated vector, called the \emph{first landing vector}:

\[
L_{j} \coloneqq \left(\sum_{s = 1}^r k_s(a_j, l_{j}) \, \bm{e}_s, \, - \bm{f}_{\text{ind}(\beta(j))}\right).
\]
\end{definition}

If $a_j$ is a discontinuity of $T$, i.e.\ if the landing time $l_j$ is zero, then $L_j = \left( 0, - \bm{f}_{\text{ind}(\beta(j))}\right)$ by definition. We call $L_j$ the first landing vector of $a_j$ to $\beta(j)$ because the following holds:

\begin{align*}
0 &= a_j + \sum_{s = 1}^r k_s(a_j, l_{j}) \gamma_s - \beta(j) \\
&= \left(\sum_{s = 1}^r k_s(a_j, l_{j}) \gamma_s - \beta(j)\right) + a_j \\
&= \langle L_j, (\gamma \, \beta) \rangle + a_j.
\end{align*}

The second type of vectors we need to consider are the \textit{critical connection vectors}, which correspond to landings of discontinuities in the orbit of $J$ onto other discontinuities (which are thus also in the orbit of $J$).

\begin{definition}[Critical connection vectors]
\label{def:cc-vec}
For each $1 \le j < N$ and $1 \le k < m_{j}^{+}$, let $q^{+}(j,k)$ be the landing time of $\beta^{+}(j,k)$ to $\beta^{+}(j,k+1)$ and let $C^{+}(j,k) \in W(r)$ be the associated vector, called the \emph{critical connection vector}:
\[
C^{+}(j,k) \coloneqq \left(\sum_{s=1}^r k_s(\beta^{+}(j,k), q^{+}(j,k)) \, \bm{e}_s, \, \bm{f}_{\text{ind}(\beta^{+}(j,k))} - \bm{f}_{\text{ind}(\beta^{+}(j,k+1))}\right).
\]
\end{definition}
The following holds by construction:
\begin{align*}
0 &= \beta^{+}(j,k) + \sum_{s=1}^r k_s(\beta^{+}(j,k), q^{+}(j,k)) \gamma_s - \beta^{+}(j,k+1) \\
&= \sum_{s=1}^r k_s(\beta^{+}(j,k), q^{+}(j,k)) \gamma_s + \beta^{+}(j,k) - \beta^{+}(j,k+1) \\
&= \langle C^{+}(j,k), (\gamma \, \beta) \rangle.
\end{align*}
Define $C^{-}(j,k) \in W(r)$ analogously for $1 \le j < N$ and $1 \le k < m_{j}^{-}$.

The third type of dynamical vectors we need to consider is the \textit{return vector}, which corresponds to the return of points $a_1^{-}, a_1^{+}, \dots, a_{N-1}^{+}, a_{N-1}^{-}$ to the interval $J$.

\begin{definition}[Return vectors]
\label{def:ret-vec}
For each $1 \le j < N$, let $r_{j}^{+}$ be the time at which $\beta^+(j,m_{j}^{+})$ lands into $J$ and let $R_{j}^{+} \in W(r)$ be the associated vector, called the \emph{return vector}:
\[
R_{j}^{+} \coloneqq \left(\sum_{s=1}^r k_s(\beta^{+}(j,m_{j}^{+}), r_{j}^{+}) \, \bm{e}_s, \, \bm{f}_{\text{ind}(\beta^{+}(j,m_{j}^{+}))}\right).
\]
\end{definition}
By definition, the following holds:
\[
\langle R_j^{+}, (\gamma \, \beta) \rangle = \sum_{s=1}^r k_s(\beta^{+}(j,m_{j}^{+}), r_{j}^{+}) \gamma_s + \beta^+(j,m_{j}^{+}) \in J.
\]
Note that also:
\begin{equation}
\label{eq:R-vec-property}
R_J(a_{j}^+) = \langle R_j^{+}, (\gamma \, \beta) \rangle 
\end{equation}
since $a_{j}$ lands on $\beta^+(j,m_{j}^{+})$. Define $r_{j}^{-}$ and $R_{j}^{-} \in W(r)$ analogously for $1 \le j < N$.

For the boundary points $x^{+} = a_0^{+}$ and $y^- = a^{-}_{N}$, the definitions of the corresponding dynamical vectors depend on whether they land on discontinuities of $T$ before returning to $J$ or not. If $a_0^{+}$ lands on a discontinuity of returning to $J$, then we may analogously as for other $a_j^{+}$, where $1 \le j < N$, define the vectors $L_0$, $C^{+}(0,k)$ and $R^{+}_0$. In the case when $a_0^{+}$ does not land on a discontinuity of $T$ before returning to $J$, we only define the return vector $R_0^{+}$ to $J$:

\[
R_0^{+} := \left(\sum_{s=1}^r k_s(a_0^{+}, r_{0}^{+}) \, \bm{e}_s, \, 0 \right),
\]
where $r_0^+$ is the return time of $a_0^+$ to $J$. Note that in this case:
\[
\langle R_0^{+}, (\gamma \, \beta) \rangle = R_J(a_{0}^+) - a_0.
\]
Analogously, if $a_{N}^{-}$ lands on a discontinuity of $T$, we may define the vectors $L_{N}, C^{-}(N,k)$ and $R^{-}_{N}$. If $a_{N}^{-}$ does not land on a discontinuity of $T$ before returning to $J$, then define $R_{N}^{-}$ analogously as for $a_0^{+}$.

Because the definitions of the dynamical vectors associated to the boundary points of the interval $J$ are different depending on whether they land on discontinuities of $T$ or not, this leads to different statements of Lemma \ref{lem:lin-dep-rj}, depending on whether these landings happen or not. For simplicity, we will assume that these boundary points land on discontinuities of $T$, because this case is more complicated, and we state Lemma \ref{lem:lin-dep-rj} with this assumption.

\begin{lemma}[Linear Dependence of Return Map Vectors]
\label{lem:lin-dep-rj}
Let $T \in \ITM(r)$
, and let $J$ be an interval component of $X$. Assume that there exist real coefficients $\alpha_{j}, \alpha_j^{+}, \alpha_j^{-}, \alpha^{+}(j,k), \alpha^{-}(j,k)$ such that:
\begin{align}
\label{eq:lin-dep-sum}
\begin{split}
&\sum_{j=0}^{N-1} \left( \sum_{k=1}^{m_j^{+}-1} \alpha^{+}(j,k) C^{+}(j,k) + \alpha_j^{+} R_j^{+} \right) \\ 
+&\sum_{j=1}^{N} \left( \sum_{k=1}^{m_j^{-}-1} \alpha^{-}(j,k) C^{-}(j,k) + \alpha_j^{-} R_j^{-} \right) \\
+& \left( \sum_{j=0}^{N} \alpha_j L_j \right) = 0.    
\end{split}
\end{align}
Then the following equalities hold:
\begin{align}
\label{eq:lin-dep-equality1}
\begin{split}
&\alpha^{+}(j,1) = \dots = \alpha^{+}(j,m_j^{+}-1) = \alpha_j^{+} \\
&=-\alpha^{-}(j+1,1) = \dots = - \alpha^{-}(j+1,m_{j+1}^{-}-1) = -\alpha_{j+1}^{-},
\end{split}
\end{align}
for all $0 \le j < N$. Moreover, $\alpha_j = \alpha^-(j,1) + \alpha^+(j,1)$ for all $1 \le j < N$.
\end{lemma}

In the case when the boundary point $a_0^+$ of $J$ does not land on a discontinuity of $T$, we need to remove $L_0$ and $C^+(0,k)$ from \eqref{eq:lin-dep-sum} and all of the coefficients for $j=0$ from the first line of \eqref{eq:lin-dep-equality1} except $\alpha_0^+$. A similar procedure should be done when $a_N^-$ does not land on a discontinuity.

Theorem \ref{thm:ae-fin-stab} will be a consequence of the following result, combined with the Characterisation of Stable Maps Theorem \ref{thm:stability=accm}.

\begin{theorem}
\label{thm:acc+m-full-m}
The set of finite types maps that do not satisfy at least one of the Matching or ACC conditions forms a measure-zero subset of $\ITM(r)$. 
\end{theorem}
\noindent
To prove this theorem, we need a lemma about the number of linear equations the parameters $(\gamma \, \beta)$ defining a finite type map $T$ need to satisfy in order to have an interval component of $X$ with $n \ge 3$ continuity intervals of the return map:

\begin{lemma}
\label{lem:rj-equation-number}
Let $T \in \ITM(r)$ be such that there is an interval component $J$ of $X$ for which the return map $R_J$ has exactly $N \ge 3$ continuity intervals. Then there are at least $N-2$ linearly independent vectors $v_1, \dots, v_{N-2}$ with integer coefficients such that defining parameters $(\gamma \, \beta)$ of $T$ satisfy $\langle v_i, (\gamma \, \beta) \rangle = 0$, for all $1 \le i \le N-2$.
\end{lemma}

\begin{proof}
Let $\sigma$ be the permutation associated to $R_J$ and let $\tau$ be its inverse. We have the following string of $N-1$ equalities:
\begin{align}
\label{eq:rj-eq}
\begin{split}
R_J(a^-_{\tau(1)}) &= R_J(a^+_{\tau(2)-1}) \\
R_J(a^-_{\tau(2)}) &= R_J(a^+_{\tau(3)-1}) \\
&\dots \\
R_J(a^-_{\tau(N-1)}) &= R_J(a^+_{\tau(N)-1}),    
\end{split}
\end{align}
where the $a_j$'s are the discontinuities of $R_J$ or the boundary points of $J$. Recall from \eqref{eq:R-vec-property} that for $j \neq 0$, we have that $R_J(a_j^+) = \langle R^{+}_{j}, (\gamma \, \beta) \rangle$. Similarly, $R_J(a_j^-) = \langle R^{-}_{j}, (\gamma \, \beta) \rangle$ for every $j \neq N$ (Definition \ref{def:ret-vec}). In the case when $a_0^+$ lands on a discontinuity before returning to $J$, we also have that $R_J(a_0^+) = \langle R^{+}_{0}, (\gamma \, \beta) \rangle$. Analogously, $R_J(a_N^-) = \langle R^{-}_{N}, (\gamma \, \beta) \rangle$ when $a_N^-$ lands on a discontinuity before returning to $J$. 
Let us first assume that both $a_0^+$ and $a_N^-$ land on discontinuities before returning to $J$. Then the equalities in \eqref{eq:rj-eq} can be written as:
\begin{align}
\label{eq:vector-eq}
\begin{split}
\langle R^{-}_{\tau(1)} - R^{+}_{\tau(2)-1}, (\gamma \, \beta) \rangle &= 0 \\
\langle R^{-}_{\tau(2)} - R^{+}_{\tau(3)-1}, (\gamma \, \beta) \rangle &= 0 \\
&\dots \\
\langle R^{-}_{\tau(N-1)} - R^{+}_{\tau(N)-1}, (\gamma \, \beta) \rangle &= 0
\end{split}
\end{align}
Assume that the $N-1$ vectors $R^{-}_{\tau(1)} - R^{+}_{\tau(2)-1}, \dots, R^{-}_{\tau(N-1)} - R^{+}_{\tau(N)-1}$ are linearly dependent with coefficients $\alpha_1^*, \dots, \alpha_{N-1}^*$:
\begin{equation}
\label{eq:lin-dep}
\sum_{j=1}^{N-1} \alpha_j^*(R^{-}_{\tau(j)} - R^{+}_{\tau(j+1)-1}) = 0.
\end{equation}
This linear dependence can be extended to the linear dependence of all vectors associated to the return map $R_J$ as follows. We set $\alpha_{\tau(j)}^- := \alpha_j^*$ and $\alpha_{\tau(j+1)-1}^+ := -\alpha_j^*$ for all $1 \le j \le N-1$, and $\alpha_{\tau(1)-1}^+ := 0, \alpha_{\tau(N)}^- := 0$. We set  $\alpha_j, \alpha^{+}(j,k), \alpha^{-}(j,k)$ all equal to zero.

Then by part \eqref{eq:lin-dep-equality1} of Lemma \ref{lem:lin-dep-rj} we have that $\alpha^-_{\tau(j)} = -\alpha^+_{\tau(j)-1}$. By definition $\alpha^+_{\tau(j)-1} = -\alpha^*_j = -\alpha^-_{\tau(j)}$, and therefore $\alpha^*_j = \alpha^-_{\tau(j)} =-\alpha^+_{\tau(j)-1} = \alpha^-_{\tau(j-1)} = \alpha^*_{j-1}$. Therefore all of the $\alpha^*_j$ are equal. Since $\alpha^*_1 = -\alpha^+_{\tau(1)-1} = 0$, we have that $\alpha_j^* = 0$ for all $1 \le j < N$. Thus in this case we get that $N-1$ linearly independent vectors.

If $a_0^+$ does not land on a discontinuity, then $\langle R^{+}_{0}, (\gamma \, \beta) \rangle = R_J(a_0) - a_0$. Since $a_0 = R_J(a^+_{\tau(1)-1})$, we have that $R_J(a_0) = \langle R^{+}_{\tau(1)-1} + R^{+}_{0}, (\gamma \, \beta) \rangle$.
Thus we need to change the equality for $j = \sigma(1)-1$ in \eqref{eq:vector-eq} to:
\begin{equation}
\langle R^{-}_{\tau(\sigma(1)-1)} - R^{+}_{0} - R^{+}_{\tau(1)-1}, (\gamma \, \beta) \rangle = 0.
\label{eq:a_0^+}
\end{equation}
If $a_N$ does not land on a discontinuity, then we analogously need to change the equality for $j = \sigma(N)$ in \eqref{eq:vector-eq} to:
\begin{equation}
\label{eq:a_N^-}
\langle R_{N}^{-} + R_{\tau(N)}^{-} - R^+_{\tau(\sigma(N)+1)-1}, (\gamma \, \beta) \rangle = 0.
\end{equation}

Assume now that $a_0^+$ lands on a discontinuity, with the case for $a_N^-$ being analogous. We now take the subset of $N-2$ vectors in \eqref{eq:vector-eq} that does not include the vector from \eqref{eq:a_0^+}. Assume that these vectors are linearly dependent with coefficients $\alpha_{j}^*$ for $2 \le j < N$. We may again extend it to all vectors associated to the return map $R_J$ by setting $\alpha_{\tau(j)}^- := \alpha_j^*$ and $\alpha_{\tau(j+1)-1}^+ := -\alpha_j^*$ for all $2 \le j \le N-1$, $\alpha^-_{\tau(\sigma(1)-1)} = \alpha^{+}_{0} = \alpha^{+}_{\tau(1)-1} = 0$, and $\alpha_j = \alpha^{+}(j,k) = \alpha^{-}(j,k) = 0 =: \alpha^*_1$. In the case when $a_N^-$ does not land on a discontinuity, we set $\alpha_{\tau(N)}^- := \alpha^*_N$ because of \eqref{eq:a_N^-}, and otherwise $\alpha_{\tau(N)}^- = 0$. 

Then again by \eqref{eq:lin-dep-equality1} from Lemma \ref{lem:lin-dep-rj}, we have that $\alpha^*_j = \alpha^-_{\tau(j)} =-\alpha^+_{\tau(j)-1} = \alpha^-_{\tau(j-1)} = \alpha^*_{j-1}$, and since $\alpha^*_1 = 0$, all the $\alpha^*_j$ are equal to zero, so in this case we get $N-2$ linearly independent vectors.
\end{proof}

The last part of the proof where we reduce the set from $N-1$ to $N-2$ vectors might appear artificial, in the sense that a different proof could still give that all of the $N-1$ vectors from \eqref{eq:vector-eq} are linearly dependent, but it is a simple exercise to check that for $N=3$, the $2$ vectors we get when neither boundary point lands on a discontinuity are linearly dependent (in fact one is the negative of the other).

\begin{proof}[Proof of Theorem \ref{thm:acc+m-full-m}]
If a finite type map $T$ violates the ACC condition, then it has a critical connection (Definition \ref{def:acc}). Thus the defining parameters $(\gamma \, \beta)$ of $T$ must satisfy $\langle C, (\gamma \, \beta) \rangle = 0$, where $C$ is the dynamical vector with integer coefficients associated to this critical connection. If $T$ violates Matching, then the return map to one of the interval components of $X$ has at least three continuity intervals. By Lemma \ref{lem:rj-equation-number}, the parameters $(\gamma \, \beta)$ of any finite type $\ITM$ on $r$ intervals for which some return map to an interval component of $X$ at least three continuity intervals must satisfy a linear equation with integer coefficients. Thus in both cases, the parameters $(\gamma \, \beta)$ are contained in a zero measure subset of $\ITM(r)$ consisting of a countable union of hyperplanes in $\ITM(r)$ defined as the orthogonal complements to vectors with integer coefficients, so the result follows.
\end{proof}

We are now ready to prove:

\begin{theorem}
\label{thm:ae-fin-stab}
In $\ITM(r)$, the set of stable maps forms a full measure subset of the set of all finite type maps.
\end{theorem}

\begin{proof}
By Theorem \ref{thm:acc+m-full-m}, the set of finite type maps that violate either ACC or Matching forms a measure zero subset of $\ITM(r)$. Thus almost every finite type map satisfies both ACC and Matching. By the Characterisation of Stable Maps (Theorem \ref{thm:stability=accm}), a map is stable if and only if it satisfies ACC and Matching, so the proof follows.
\end{proof}

We say that a map $T$ `corresponds to an irrational rotation' if it is of finite type and the return map to every interval component of $X$ is isomorphic to an irrational circle rotation. A consequence of Theorem \ref{thm:acc+m-full-m} is the following refinement of the Boshernitzan--Kornfeld Conjecture: Almost every $\ITM$ corresponds to an irrational circle rotation. More precisely, we have the following result:

\begin{theorem}
\label{thm:bk->irr-rot}
Assume that the Boshernitzan--Kornfeld Conjecture holds. Then the set of all $\ITM$s that correspond to an irrational rotation forms a full measure subset of $\ITM(r)$.
\end{theorem}

\begin{proof}
By Theorem \ref{thm:acc+m-full-m}, a full measure subset of finite type maps satisfies ACC and Matching. The return map to an interval component of $X$ of a map $T$ that satisfies ACC and Matching is either isomorphic to a rotation or the identity. If some return map to an interval component of $X$ is isomorphic to a rational rotation or is the identity, then $T$ must have a periodic point. Let $x$ be a periodic point of $T$ with minimal period $p$, so that $T^p(x) = x$. Then the defining parameters $(\gamma \, \beta)$ of $T$ satisfy $\langle C, (\gamma \, \beta) \rangle = 0$, where $C := (k_1(x,p), \dots, k_r(x,p), 0, \dots, 0)$. Thus the set of all maps $T$ for which some return map to an interval component of $X$ is isomorphic to a rational rotation or is the identity has measure zero in $\ITM(r)$. Thus the set of all finite type maps for which every return map to an interval component of $X$ is isomorphic to an irrational circle rotation is a full measure subset of the set of finite type maps. Thus if the Boshernitzan--Kornfeld Conjecture holds, and the set of all finite type maps has full measure in $\ITM(r)$, then so does the set of all maps corresponding to irrational rotations, and the result follows.
\end{proof}

\addcontentsline{toc}{section}{References}
\printbibliography

\end{document}